\documentclass{article}
\usepackage[utf8]{inputenc}
\usepackage{arxiv}
\usepackage{hyperref}       % hyperlinks
\usepackage{url}            % simple URL typesetting
\usepackage{amsmath,amssymb,amsthm,graphicx}
\usepackage{booktabs}       % professional-quality tables
\usepackage{amsfonts}       % blackboard math symbols
\usepackage{nicefrac}       % compact symbols for 1/2, etc.
\usepackage{microtype}      % microtypography
\usepackage{dsfont}
\usepackage{color}
\usepackage{rotating}

\newtheoremstyle{thm}% name
{9pt}%      Space above, empty = `usual value'
{9pt}%      Space below
{\itshape}% Body font
{}%         Indent amount (empty = no indent, \parindent = para indent)
{\bfseries}% Thm head font
{.}%        Punctuation after thm head
{ }% Space after thm head: \newline = linebreak
{}%         Thm head spec
\theoremstyle{thm}
\newtheorem{theorem}{Theorem}[section]
\newtheorem{lemma}[theorem]{Lemma}
\newtheorem{corollary}[theorem]{Corollary}

\newtheoremstyle{def}% name
{9pt}%      Space above, empty = `usual value'
{9pt}%      Space below
{}% Body font
{}%         Indent amount (empty = no indent, \parindent = para indent)
{\bfseries}% Thm head font
{.}%        Punctuation after thm head
{ }% Space after thm head: \newline = linebreak
{}%         Thm head spec
\theoremstyle{def}

\newtheorem{remark}[theorem]{Remark}

\DeclareMathOperator{\sgn}{sgn}
\DeclareMathOperator{\arcsine}{arcsine}
\newcommand{\R}{\mathbb{R}} % reelle
\newcommand{\E}{\mathbb{E}} %Erwartungswert
\newcommand{\PP}{\mathbb{P}} %Wahrscheinlichkeit
\newcommand{\fse}{\overset{a.s.}{\longrightarrow}}
    \def\cd{\stackrel{\mathcal{D}}{\longrightarrow}}
    \def\cp{\stackrel{\mathcal{\PP}}{\longrightarrow}}
\renewcommand{\footnoterule}{%
	\kern -3.5pt
	\hrule width \textwidth height 1pt
	\kern 3.5pt
}

\renewcommand{\footnoterule}{%
	\kern -3.5pt
	\hrule width \textwidth height 1pt
	\kern 3.5pt
}

\makeatletter
\def\blfootnote{\xdef\@thefnmark{}\@footnotetext}
\makeatother

\title{Cauchy or not Cauchy? New goodness-of-fit tests for the Cauchy distribution}
\author{Bruno Ebner\\
 Institute of Stochastics, \\
Karlsruhe Institute of Technology (KIT), \\
Englerstr. 2, D-76133 Karlsruhe. \\
\href{mailto:Bruno.Ebner@kit.edu}{Bruno.Ebner@kit.edu}\\
\And
Lena Eid\\
Institute of Stochastics, \\
Karlsruhe Institute of Technology (KIT), \\
Englerstr. 2, D-76133 Karlsruhe. \\
\And
Bernhard Klar\\
Institute of Stochastics, \\
Karlsruhe Institute of Technology (KIT), \\
Englerstr. 2, D-76133 Karlsruhe. \\
\href{mailto:Bernhard.Klar@kit.edu}{Bernhard.Klar@kit.edu}
}
\date{\today}

\begin{document}

\maketitle


\begin{abstract}
We introduce a new characterization of the Cauchy distribution and propose a class of goodness-of-fit tests to the Cauchy family. The limit distribution is derived in a Hilbert space framework under the null hypothesis and under fixed alternatives. The new tests are consistent against a large class of alternatives. A comparative Monte Carlo 
simulation study shows that the test is competitive to the state of the art procedures, and we apply the tests to log-returns of cryptocurrencies.
\end{abstract}

\section{Introduction}\label{sec:Intro}
In this article we dedicate our studies to answer the question of whether a data set of univariate real numbers belongs to the famous family of Cauchy distributions. The Cauchy distribution is undoubtedly the standard example for a distribution without existing mean value and was studied in the mathematical world for more than 300 years, having wide applicability in diverse fields ranging from modeling resonances in physics (then often called Lorentz distribution) to cryptocurrencies in finance, see \cite{S:2020}. It is also known as the Breit-Wigner distribution, for an extensive historical overview, see \cite{S:1974}. To be precise, we write shorthand C$(\alpha,\beta)$, $\alpha \in\R$, $\beta > 0$, for the Cauchy distribution with location parameter $\alpha$ and scale parameter $\beta$, having density
\begin{equation*}
f(x,\alpha,\beta) = \frac{1}{\pi}\frac{\beta}{\beta^2+(x-\alpha)^2},\quad x\in\R.
\end{equation*}
For a detailed discussion on this family, see \cite{JKB:1994}, Chapter 16. The Cauchy distribution is a so called heavy tailed distribution and is a member of the stable distributions, see \cite{N:2020}. Note that $X \sim$ C$(\alpha, \beta)$ if, and only if, $\frac{X-\alpha}{\beta} \sim$ C$(0,1)$ and hence the Cauchy distribution belongs to the location-scale family of distributions. In the following we denote the family of Cauchy distributions by $\mathcal{C}:=\{\mbox{C}(\alpha,\beta):\,\alpha\in\R,\beta>0\}$, a family of distributions which is closed under translation and rescaling. We test the composite hypothesis
\begin{equation}\label{eq:H0}
H_0:\;\mathbb{P}^X\in\mathcal{C}
\end{equation}
against general alternatives on the basis of independent identical copies $X_1,\ldots,X_n$ of $X$. This testing problem has been considered in the literature: \cite{GH:2000} propose a test procedure based on the empirical characteristic function and \cite{MT:2005} extend this test by considering alternative estimation methods. More recently, \cite{MZ:2017} propose to use the likelihood ratio as in \cite{Z:2002} as well as the Kullback-Leibler distance, an idea that is extended in \cite{MZ:2019}. A quantile based method is proposed in \cite{R:2001} and compared to the classical omnibus procedures. An empirical power study of goodness-of-fit tests for the Cauchy model based on the empirical distribution function as the Kolmogorov-Smirnov test, the Cram\'{e}r-von Mises test, the Kuiper test, the Anderson Darling and the Watson test is found in \cite{ODYM:2001}. In \cite{L:2005}, two tests for the standard Cauchy distribution based on characterizations are given; however, they are not designed for the composite hypothesis \eqref{eq:H0}.

The novel procedure is based on the following new characterization of the standard Cauchy distribution.
\begin{theorem}\label{thm:Char}
Let $X$ be a random variable with absolutely continuous density $p$ and $\E\left[ \frac{|X|}{1+X^2}\right] < \infty$. Then $X$ has a Cauchy Distribution C$(0,1)$ if, and only if 
\begin{equation}\label{eq:meanchar}
\E \left[\left(it-\frac{2X}{1+X^2}\right)\exp(itX)\right]= 0
\end{equation}
holds for all $t\in \R$, where $i$ denotes the imaginary unit.
\end{theorem}
\begin{proof}
For $X\sim\mbox{C}(0,1)$ direct calculation shows the assertion. Let $X$ be a random variable with absolutely continuous density function $p$ such that 
\begin{equation*}
\E \left[\left(it-\frac{2X}{1+X^2}\right)\exp(itX)\right]= 0
\end{equation*}
holds for all $t\in \R$. Note that since $-it\E[\exp(itX)]$ is the Fourier-Stieltjes transform of the derivative of $p$ we have
\begin{equation*}
0=\E \left[\left(it-\frac{2X}{1+X^2}\right)\exp(itX)\right]=\int_{-\infty}^\infty \left(-p'(x)-\frac{2x}{1+x^2}p(x)\right)\exp(itx)\,\mbox{d}x
\end{equation*}
for all $t\in\R$. By properties of the Fourier-Stieltjes transform we hence note that $p$ must satisfy the ordinary differential equation
\begin{equation*}
p'(x)+\frac{2x}{1+x^2}p(x)=0
\end{equation*}
for almost all $x\in\R$. The only solution satisfying $\int_{-\infty}^\infty p(x)\mbox{d}x=1$ is $p(x)=f(x,0,1)$, $x\in\R$, and $X\sim\mbox{C}(0,1)$ follows.
\end{proof}
\begin{remark}
Note that the characterization in Theorem \ref{thm:Char} is related to the spectral representation of the Stein operator of the so called density approach, for details on Stein operators, see \cite{SReview:2021}.
\end{remark}

\subsection{A new class of goodness of fit tests for the Cauchy distribution}
Note that the testing problem under discussion is invariant with respect to transformations of the kind $x\to ax+b, \, x\in\R$, where $a\in\R$ and $b>0$. Consequently, a decision in favor or against $H_0$ should be the same for  $X_1,\ldots,X_n$ and  $aX_1+b,\ldots,aX_n+b$. This goal is achieved if the test statistic $T_n$, say, is based on the scaled residuals $Y_{n,1},...,Y_{n,n}$, given by
\begin{equation}
\label{Y}
Y_{n,j}=\frac{X_j-\widehat{\alpha}_n}{\widehat{\beta}_n}, \quad j=1,\ldots,n.
\end{equation}
Here, $\widehat{\alpha}_n=\widehat{\alpha}_n(X_1,...,X_n)$ and $\widehat{\beta}_n=\widehat{\beta}_n(X_1,...,X_n)$ denote consistent estimators of $\alpha\in\R$ and $\beta>0$ such that 
\begin{eqnarray}
\label{eq:alpha}
\widehat{\alpha}_n(aX_1+b,...,aX_n+b)&=&a\widehat{\alpha}_n(X_1,...,X_n)+b,\\\label{eq:beta}
\widehat{\beta}_n(aX_1+b,...,aX_n+b)&=&a\widehat{\beta}_n(X_1,...,X_n),
\end{eqnarray}
holds for each $a>0$ and $b\in\R$. By (\ref{eq:alpha}) and (\ref{eq:beta}) it is easy to see that $Y_{n,j}$, $j=1,\ldots,n,$ do not depend on the location nor on the scale parameter. Hence, $T_n$ has the property
$T_n(aX_1+b,\ldots,aX_n+b)=T_n(X_1,\ldots,X_n)$, and we may and do assume $\alpha=0$ and $\beta=1$ in the following. Specifically, we choose the test statistic
\begin{equation}
\label{eq:statistic}
T_{n}=n\int_{-\infty}^{\infty}\Big\vert\frac{1}{n}\sum_{j=1}^n\Big(it-\frac{2Y_{n,j}}{1+Y_{n,j}^2}\Big)e^{itY_{n,j}}\Big\vert^2\omega(t)\;\mbox{d}t,
\end{equation}
which is the weighted $L^2$-distance from (\ref{eq:meanchar}) to 0. Here, $\omega:\R\rightarrow\R$ denotes a weight function such that
\begin{equation}
\label{omega}
\omega(t)=\omega(-t),\ t\in\R,\text{ and }\ \int_{-\infty}^{\infty}\omega(t)\;\mbox{d}t<\infty,
\end{equation}
and $|\cdot|$ is the complex absolute value. For the particular choice $\omega(t)=\omega_a(t)=\exp(-a|t|)$, $t\in\R$, $a>0$, of the weight function, we get the integration-free, numerical stable formula
\begin{align}
\label{eq:EF}
T_{n,a}=\frac{1}{n}\sum_{j,k=1}^n\bigg(&\frac{8aY_{n,j}Y_{n,k}}{(1+Y_{n,j}^2)(1+Y_{n,k}^2)((Y_{n,j}-Y_{n,k})^2+a^2)}
-\frac{16aY_{n,j}(Y_{n,j}-Y_{n,k})}{(1+Y_{n,j}^2)((Y_{n,j}-Y_{n,k})^2+a^2)^2}\nonumber\\\
&+\frac{4a^3-12a(Y_{n,j}-Y_{n,k})^2}{((Y_{n,j}-Y_{n,k})^2+a^2)^3}\bigg),
\end{align}
and hence a whole family of tests depending on the so called tuning parameter $a>0$. 
The next result reveals the limit behavior of $T_{n,a}$ for $a\rightarrow 0$ and $a\rightarrow\infty$.

\begin{theorem}\label{thmlimit} 
For fixed $n$, we have
\begin{align} \label{stat:lim1}
 \lim_{a\to 0}  a \bigg( T_{n,a}-\frac{4}{a^3} \bigg) &= \frac{8}{n} \sum_{j=1}^n \frac{Y_{n,j}^2}{(1+Y_{n,j}^2)^2},
\end{align}
and
\begin{align} \label{stat:lim2}
\lim_{a\to\infty} aT_{n,a} &= \frac{8}{n} \Bigg( \sum_{j=1}^n \frac{Y_{n,j}}{1+Y_{n,j}^2} \Bigg)^2.
\end{align}
\end{theorem}

\begin{proof}
Splitting the sum in (\ref{eq:EF}) in a diagonal and a non-diagonal part results in 
\begin{align*}
T_{n,a}= \frac{1}{n}\sum_{j,k=1}^n R_{j,k,a}
=\frac{1}{n}\sum_{j=1}^n R_{j,j,a} +\frac{1}{n}\sum_{j\neq k} R_{j,k,a}
=T_{n,a}^{d}+T_{n,a}^{nd},
\end{align*}
say. Since $\lim_{a\to 0}T_{n,a}^{nd}=0$, we obtain
\begin{align*}
\lim_{a\to 0}  aT_{n,a} &= \lim_{a\to 0}  aT_{n,a}^{d} 
= \frac{8}{n} \sum_{j=1}^n \left\{ \frac{Y_{n,j}^2}{(1+Y_{n,j}^2)^2} + \frac{4}{a^2} \right\},
\end{align*}
and the first assertion follows.

Observe that $T_{n,a}  =  \int_0^\infty g(t) \exp(-at) \: dt$,
where $g(u) = n (h_1^2(t)+h_2^2(t))$, and $h_1$ and $h_2$ are defined by
\begin{align*}
h_1(t) &=  \frac{1}{n} \sum_{j=1}^n \bigg\{ \bigg(\frac{2Y_{n,j}}{1+Y_{n,j}^2}+t\bigg) \cos(tY_{n,j})
 + \bigg(t-\frac{2Y_{n,j}}{1+Y_{n,j}^2}\bigg) \sin(tY_{n,j}) \bigg\}, \\
h_2(t) &=  \frac{1}{n} \sum_{j=1}^n \bigg\{ \bigg(\frac{2Y_{n,j}}{1+Y_{n,j}^2}-t\bigg) \cos(tY_{n,j})
 + \bigg(t+\frac{2Y_{n,j}}{1+Y_{n,j}^2}\bigg) \sin(tY_{n,j}) \bigg\}.
 \end{align*}
Since
\begin{align*}
\lim_{t\to 0} h_1(t) &= \lim_{t\to 0} h_2(t) 
\ = \frac{2}{n} \sum_{j=1}^n \frac{Y_{n,j}}{1+Y_{n,j}^2},
 \end{align*}
the second assertion follows from an Abelian theorem on Laplace transforms (see \cite{W:1959}, p.182).
\end{proof}

The rest of the paper is organized as follows. In Section 2, the limit distribution of the test statistic $T_{n,a}$ is derived in a Hilbert space framework under the null hypothesis and under fixed alternatives. Furthermore, we derive the limit distribution of $T_{n,0}$. Consistency of the new tests against a large class of alternatives is shown in Section 3. An extensive Monte Carlo simulation study in Section 4 shows that the test is competitive to the state of the art procedures. Finally, in Section 4, the tests are applied to log-returns of cryptocurrencies.

\section{Limit distribution under the null hypothesis}
The asymptotic theory is derived in the Hilbert space $\mathbb{H}$ of measurable, square integrable functions $\mathbb{H}=L^2(\mathbb{R},\mathcal{B}, \omega(t)\mbox{d}t)$, where $\mathcal{B}$ is the Borel-$\sigma$-field of $\mathbb{R}$. Notice that the functions figuring within the integral in the definition of $T_{n}$ are $(\mathcal{A} \otimes \mathcal{B}, \mathcal{B})$-measurable random elements of $\mathbb{H}$. We denote by
\begin{equation*}
	\|f\|_{\mathbb{H}} = \left( \int_{-\infty}^\infty \big|f(t)\big|^2 \, \omega(t) \, \mathrm{d}t \right)^{1/2}, \qquad \langle f, g \rangle_{\mathbb{H}}=\int_{-\infty}^\infty f(t)g(t) \, \omega(t) \, \mathrm{d}t
\end{equation*}
the usual norm and inner product in $\mathbb{H}$.  In the following, we assume that the estimators $\widehat{\alpha}_n$ and $\widehat{\beta}_n$ allow linear representations 
\begin{align}
    \sqrt{n} \widehat{\alpha}_n&= \frac{1}{\sqrt{n}} \sum_{j = 1}^n \psi_1 (X_{j}) + o_{\mathbb{P}}(1),\label{eq:psi_11}\\
    \sqrt{n} (\widehat{\beta}_n - 1) &= \frac{1}{\sqrt{n}} \sum_{j = 1}^n \psi_2 (X_{j}) + o_{\mathbb{P}}(1),\label{eq:psi_21}
\end{align}
where $o_\mathbb{P}(1)$ denotes a term that converges to 0 in probability, and $\psi_1$ und $\psi_2$ are measurable functions with
\begin{equation*}
    \mathbb{E}[\psi_1 (X_{1})] = \mathbb{E}[\psi_2 (X_{1})] = 0, \quad\mbox{and}\quad \mathbb{E}[\psi_1^2 (X_{1})] < \infty,\;  \mathbb{E}[\psi_2^2 (X_{1})] < \infty, \label{eq:psi2}
\end{equation*}
see Remark \ref{rem:est} for examples of estimation procedures satisfying these assumptions. By the symmetry of the weight function $\omega(\cdot)$ straightforward calculations show
\begin{equation*}
    T_n=\int_{-\infty}^\infty Z_n^2(t)\,\omega(t)\mbox{d}t,
\end{equation*}
where
\begin{equation}\label{eq:Zn}
    Z_n(t)=\frac{1}{\sqrt{n}}\sum_{j=1}^n\Big(\frac{2Y_{n,j}}{1+Y_{n,j}^2}+t\Big)\cos(tY_{n,j})+\Big(t-\frac{2Y_{n,j}}{1+Y_{n,j}^2}\Big)\sin(tY_{n,j}),\ \ t\in\R,
\end{equation}
is a real-valued $(\mathcal{A} \otimes \mathcal{B}, \mathcal{B})$-measurable random element of $\mathbb{H}$.

\begin{theorem}\label{thm:asympVer}
Let $X_1,...,X_n$ be i.i.d. $\mbox{C}(0,1)$ distributed random variables and $Z_n$ be the random element of $\mathbb{H}$ as in \eqref{eq:Zn}. Then there exists a centred Gaussian random process $\mathcal{Z}$ in $\mathbb{H}$ with covariance kernel
\begin{align*}
K(s,t)&=\frac{1}{2}\big(s^2+t^2+\vert s-t\vert+1\big)e^{-\vert s-t\vert}\\
&-\frac{1}{2}(t^2+\vert t\vert+1)e^{-\vert t\vert}\E\bigg[\bigg(\bigg(\frac{2X_1}{1+X_1^2}+s\bigg)\cos(sX_1)+\bigg(s-\frac{2X_1}{1+X_1^2}\bigg)\sin(sX_1)\bigg)\psi_1(X_1)\bigg]  \\
&+\frac{1}{2}t(\vert t\vert+1)e^{-\vert t\vert}\E\bigg[\bigg(\bigg(\frac{2X_1}{1+X_1^2}+s\bigg)\cos(sX_1)+\bigg(s-\frac{2X_1}{1+X_1^2}\bigg)\sin(sX_1)\bigg)\psi_2(X_1)\bigg]\\
&-\frac{1}{2}(s^2+\vert s\vert+1)e^{-\vert s\vert}\E\bigg[\bigg(\bigg(\frac{2X_1}{1+X_1^2}+t\bigg)\cos(tX_1)+\bigg(t-\frac{2X_1}{1+X_1^2}\bigg)\sin(tX_1)\bigg)\psi_1(X_1)\bigg]\\
&+\frac{1}{2}s(\vert s\vert+1)e^{-\vert s\vert}\E\bigg[\bigg(\bigg(\frac{2X_1}{1+X_1^2}+t\bigg)\cos(tX_1)+\bigg(t-\frac{2X_1}{1+X_1^2}\bigg)\sin(tX_1)\bigg)\psi_2(X_1)\bigg]\\
&+\frac{1}{4}(s^2+\vert s\vert+1)(t^2+\vert t\vert+1)e^{-\vert s\vert-\vert t\vert}\E[\psi_1^2(X_1)]+\frac{1}{4}s(\vert s\vert+1)t(\vert t\vert+1)e^{-\vert s\vert-\vert t\vert}\E[\psi_2^2(X_1)],\quad s,t\in\R,
\end{align*}
such that $Z_n\overset{D}{\longrightarrow}\mathcal{Z}$ in $\mathbb{H}$ as $n\rightarrow\infty$. 
\end{theorem}
A proof of Theorem \ref{thm:asympVer} is found in Appendix \ref{app:proofasy1}. An application of the continuous mapping theorem states the following corollary.

\begin{corollary}\label{cor:H0asy}
We have as $n\rightarrow\infty$ 
\[T_n\overset{D}{\longrightarrow}\int_{-\infty}^\infty \mathcal{Z}^2(t)\omega(t)dt=\Vert \mathcal{Z}\Vert_{\mathbb{H}}^2.\]
\end{corollary}
It is well known that the distribution of $\Vert \mathcal{Z}\Vert_{\mathbb{H}}^2$ is that of $\sum_{j=1}^\infty \lambda_jN_j^2$, where $N_1,N_2,\ldots$ are i.i.d. standard normal random variables and $(\lambda_j)$ is a decreasing sequence of positive eigenvalues of the integral operator
\begin{equation*}
    \mathcal{K}g(s)=\int_{-\infty}^\infty K(s,t)g(t)\omega(t)\,\mbox{d}t.
\end{equation*}
Due to the complexitiy of the covariance kernel $K$ it seems hopeless to solve the integral equation $\mathcal{K}g(s)=\lambda g(s)$ and find explicit values of $\lambda_j$, $j\ge1$. For a numerical approximation method we refer to Subsection 3.3 in \cite{MT:2005}. Note that 
\begin{equation}
    \E\Vert \mathcal{Z}\Vert_{\mathbb{H}}^2=\int_{-\infty}^\infty K(t,t)\omega(t)\,\mbox{d}t
\end{equation}
and 
\begin{equation}
    \mbox{Var}\Vert \mathcal{Z}\Vert_{\mathbb{H}}^2=2\int_{-\infty}^\infty\int_{-\infty}^\infty K^2(s,t)\omega(t)\omega(s)\,\mbox{d}t\mbox{d}s
\end{equation}
can be derived for specific estimation procedures and weight functions. Such results in the theory of goodness-of-fit tests for the Cauchy family are sparse, for some explicit formulae for mean values, see \cite{GH:2000} and \cite{MT:2005}.

\begin{remark}\label{rem:est} The generality of Theorem \ref{thm:asympVer} in view of the linear representations of the estimators and hence dependence on the functions $\psi_1$ and $\psi_2$ leads to explicit covariance kernels for different parameter estimation procedures. To estimate $\alpha$ and $\beta$ we choose the following location and scale estimators $\widehat{\alpha}_n$ and $\widehat{\beta}_n$, which all satisfy \eqref{eq:alpha} and \eqref{eq:beta} respectively. For a compact notation, we write $\psi(x)=(\psi_1(x),\psi_2(x))^\top$ and $I_2$ for the 2-dimensional identity matrix. Some derivations were partially provided by the computer algebra system Maple, see \cite{Maple2019}.
\begin{enumerate}
\item \textbf{Median and interquartile-distance estimators:} Let $\xi_p$, $p\in(0,1)$, denote the $p$-quantile of the underlying distribution $F$, $\widehat{\xi}_{p,n}$ the sample $p$-quantile, and $X_{(1)}\le\cdots\le X_{(n)}$ the order statistics of $X_1,\ldots,X_n$. With $\lfloor\cdot\rfloor$ denoting the floor function, let
\begin{equation}\label{eq:medest}
\widehat{\alpha}_n=\begin{cases}
\frac{1}{2}(X_{(\frac{n}{2})}+X_{(\frac{n}{2}+1)}), & \text{if}\,n\,\mbox{even,}\\
X_{(\lfloor\frac{n}{2}\rfloor+1)}, & \text{otherwise,}
\end{cases}
\end{equation}
be the unbiased empirical median and 
\begin{equation}\label{eq:IQRest}
\widehat{\beta}_n=\frac{1}{2}(\widehat{\xi}_{\frac{3}{4},n}-\widehat{\xi}_{\frac{1}{4},n})
\end{equation}
the half-interquartile range (iqr) of the sample. Under mild regularity conditions $\widehat{\alpha}_n$ and  $\widehat{\beta}_n$ are consistent estimators of $\alpha$ and $\beta$. Display (3.3) of \cite{GH:2000} then gives the so-called Bahadur representations (see Theorem 2.5.1 in \cite{S:1980}) with
\[\psi_1(x)=\pi\left(\frac{1}{2}-\mathbf{1}\lbrace x\leq 0\rbrace\right),\quad \mbox{and}\quad \psi_2(x)=\pi\left(\frac{1}{2}-\mathbf{1}\lbrace-1\leq x\leq 1\rbrace\right),\ x\in\R.\]
It is easy to see that $\E[\psi_1(X_1)]=\E[\psi_2(X_1)]=0$ and $\E[\psi(X_1)\psi(X_1)^\top]=\frac{\pi^2}{4}I_{2}$ holds.
With these representations, we get the covariance kernel of $\mathcal{Z}$ in Theorem \ref{thm:asympVer} 
\begin{align*}
\label{KMIS}
K_{MIQ}(s,t)=&\frac{1}{2}\big(s^2+t^2+\vert s-t\vert+1\big)e^{-\vert s-t\vert}\\
&+\Big[t\big(\vert t\vert+1\big)\big(2J_1(s)-sJ_2(s)\big)-\big(t^2+\vert t\vert+1\big)\big(\frac{s}{2}J_3(s)+J_4(s)\big)\Big]e^{-\vert t\vert}\\
&+\Big[s\big(\vert s\vert+1\big)\big(2J_1(t)-tJ_2(t)\big)-\big(s^2+\vert s\vert+1\big)\big(\frac{t}{2}J_3(t)+J_4(t)\big)\Big]e^{-\vert s\vert}\\
&+\frac{\pi^2}{16}\Big[\big(s^2+\vert s\vert+1\big)\big(t^2+\vert t\vert+1\big)+st\big(\vert s\vert+1\big)\big(\vert t\vert+1\big)\Big]e^{-\vert s\vert-\vert t\vert},\quad s,t\in\R,
\end{align*}
where 
\begin{equation*}
J_1(t)=\int_0^1\frac{x\sin(tx)}{(1+x^2)^2}dx, \ 
J_2(t)=\int_0^1\frac{\cos(tx)}{1+x^2}dx, \ 
J_3(t)=\int_0^\infty\frac{\sin(tx)}{1+x^2}dx, \
J_4(t)=\int_0^\infty\frac{x\cos(tx)}{(1+x^2)^2}dx.
\end{equation*}
Direct calculations of integrals using the weight function $\omega_a(t)$ of the introduction lead to 
\begin{eqnarray*}
    \E\|\mathcal{Z}\|_{\mathbb{H}}^2&=&\int_{-\infty}^\infty K_{MIQ}(t,t)\omega_a(t)\,\mbox{d}t\\
    &=&(8 \left( a+2 \right) ^{5} \left( 1+a
 \right) ^{3}{a}^{3} \left( {a}^{2}+2\,a+2 \right) ^{3})^{-1}\left[\left( {\pi}^{2}-8 \right) {a}^{16}+ \left( 19\,{
\pi}^{2}-152 \right) {a}^{15}\right.\\&&+ \left( 173\,{\pi}^{2}-1368 \right) {a}^
{14}+ \left( 1003\,{\pi}^{2}-7720 \right) {a}^{13}+ \left( 4126\,{\pi}
^{2}-30192 \right) {a}^{12}\\&&+ \left( 12594\,{\pi}^{2}-84304 \right) {a}
^{11}+ \left( 29128\,{\pi}^{2}-163520 \right) {a}^{10}+ \left( 51460\,
{\pi}^{2}-188832 \right) {a}^{9}\\&&+ \left( 69320\,{\pi}^{2}-8256
 \right) {a}^{8}+ \left( 70296\,{\pi}^{2}+457664 \right) {a}^{7}+
 \left( 52176\,{\pi}^{2}+1025920 \right) {a}^{6}\\&&+ \left( 26848\,{\pi}^
{2}+1323264 \right) {a}^{5}+ \left( 8576\,{\pi}^{2}+1151488 \right) {a
}^{4}+ \left( 1280\,{\pi}^{2}+693248 \right) {a}^{3}\\&&\left.+280576\,{a}^{2}+
69632\,a+8192\right].
\end{eqnarray*}

\item \textbf{Maximum likelihood estimators:} \cite{MT:2005} show in Lemma A.1 that for the maximum-likelihood estimator $\widehat{\alpha}_n$ and $\widehat{\beta}_n$ in the Cauchy family the linear representations are given by
\[\psi_1(x)=\frac{4x}{1+x^2},\quad \mbox{and}\quad\psi_2(x)=\frac{2(x^2-1)}{1+x^2},\;x\in\R.\]
Again, straightforward calculations show $\E[\psi_1(X_1)]=\E[\psi_2(X_1)]=0$ and $\E[\psi(X_1)\psi(X_1)^\top]=2 I_{2}$. Note that there are no closed form expressions for the estimators, such that the log-likelihood equations have to be solved numerically. This leads to the covariance kernel of $\mathcal{Z}$ in Theorem \ref{thm:asympVer} given by
\begin{align*}
\label{KMLS}
K_{ML}(s,t) = & \frac{1}{2}\big(s^2+t^2+\vert s-t\vert+1\big)e^{-\vert s-t\vert}-\frac{1}{2}(t^2+\vert t\vert+1)(s^2+\vert s\vert +1)e^{-\vert t\vert-\vert s\vert}\\&-\frac{1}{2}t(\vert t\vert+1)s(\vert s\vert+1)e^{-\vert t\vert-\vert s\vert},\quad s,t\in\R.
\end{align*}
With this explicit formula for the covariance kernel, we compute for the weight function $w_a(t)$ in the introduction
\begin{equation*}
    \E\|\mathcal{Z}\|_{\mathbb{H}}^2=\int_{-\infty}^\infty K_{ML}(t,t)\omega_a(t)\,\mbox{d}t={\frac {8\,{a}^{4}+80\,{a}^{3}+352\,{a}^{2}+320\,a+128}{{a}^{3}
 \left( a+2 \right) ^{5}}}
\end{equation*}
and
\begin{eqnarray*}
\mbox{Var}\|\mathcal{Z}\|_{\mathbb{H}}^2&=&2\int_{-\infty}^\infty\int_{-\infty}^\infty K_{ML}^2(s,t)\omega_a(s)\omega_a(t)\,\mbox{d}t\mbox{d}s\\&=&
\left(\frac12{a}^{5}
 \left( a+1 \right) ^{7} \left( a+2 \right) ^{10}\right)^{-1}\left[10\,{a}^{14}+270\,{a}^{13}+3521\,{a}^{12}+27987\,{a}^{11}+
146819\,{a}^{10}\right.\\&&+510582\,{a}^{9}+1194078\,{a}^{8}+1914216\,{a}^{7}+
2134432\,{a}^{6}+1671456\,{a}^{5}+928192\,{a}^{4}\\&&\left.+369792\,{a}^{3}+
104192\,{a}^{2}+18432\,a+1536\right].
\end{eqnarray*}

\item \textbf{Equivariant integrated squared error estimator:} In \cite{MT:2005} propose an equivariant version of the minimal integrated squared error estimator introduced by \cite{BM:2001} for the Cauchy distribution. The estimators are derived by minimization of the weighted $L^2$-distance 
\begin{equation*}
I(\alpha,\beta)=\int_{-\infty}^{\infty}\vert\varphi_n(t;\alpha,\beta)-e^{-\vert t\vert}\vert^2\omega(t)dt,
\end{equation*}
where $\varphi_n(t;\alpha,\beta)=\frac{1}{n}\sum_{j=1}^n \exp\Big(it\frac{X_j-\alpha}{\beta}\Big)$, $t\in\R$, is the empirical characteristic function of a distribution from the location scale family in dependence of the parameters. As weight function the authors chose $\omega_\nu(t)=\exp(-\nu|t|)$, $t\in\R$, $\nu>0$, and hence get families of estimators $\widehat{\alpha}_{n,\nu}$ and $\widehat{\beta}_{n,\nu}$ in dependence of $\nu$. Note that the optimization problem has to be solved numerically and no closed-form formula is known for the estimators. We denote this class of estimators hereinafter EISE. Lemma A.1 in \cite{MT:2005} provides the linear asymptotic representations
\begin{equation*}
\psi_1(x,\nu)=(\nu+1)(\nu+2)^3\frac{x}{((\nu+1)^2+x^2)^2}  \ \ \mbox{and} \ \ \psi_2(x,\nu)=\frac{1}{2}(\nu+2)-\frac{1}{2}(\nu+2)^3\frac{(\nu+1)^2-x^2}{((\nu+1)^2+x^2)^2}
\end{equation*}
for $x\in\R$, which lead for every $\nu>0$ to the covariance kernel of $\mathcal{Z}$ in Theorem \ref{thm:asympVer}
\begin{align*}
\label{KEISE}
K_{EISE}(s,t;\nu)&= \frac{1}{2}\big(s^2+t^2+\vert s-t\vert+1\big)e^{-\vert s-t\vert}\\
&-\frac{1}{2}(t^2+\vert t\vert+1)(\nu+1)\bigg(\frac{2\vert s\vert(\nu+2)\nu^3+(\nu+1)(1-\vert s\vert)+\vert s\vert+3}{\nu^3}\bigg)e^{-\vert s\vert-\vert t\vert}\\
&+\frac{1}{2}(t^2+\vert t\vert+1)\bigg(\frac{(\nu+1)(\vert s\vert(\nu+1)^2+3(\nu+1)-\vert s\vert)-1}{\nu^3}\\
&\ \ \ \ \ \ +(\nu+1)(s^2(\nu+2)(\nu+1)+2\vert s\vert(\nu+2)+s^2)\bigg)e^{-(\nu+1)\vert s\vert-\vert t\vert}\\
&-\frac{1}{2}t(\vert t\vert+1)(\nu+1)\bigg(\frac{s(\nu+2)^3(\nu+1)}{2\nu^2}+s(\nu+1)^2+s-4\sgn(s)\bigg)e^{-\vert s\vert-\vert t\vert}\\
&+\frac{1}{2}t(\vert t\vert+1)\bigg(\frac{s(\nu+2)^3(\vert s\vert(\nu+1)^3-3(\nu+1)^2+\vert s\vert(\nu+1)+1)}{4(\nu+1)\nu^2}\\
&\ \ \ \ \ \ -s(\nu+1)^2+s-4\sgn(s)(\nu+1)\bigg)e^{-\vert s\vert-\vert t\vert}\\
&+\frac{1}{2\pi}t(\vert t\vert+1)(\nu+2)^3\big(sJ_1(s)-2J_2(s)\big)e^{-\vert t\vert}\\
&-\frac{1}{2}(s^2+\vert s\vert+1)(\nu+1)\bigg(\frac{2\vert t\vert(\nu+2)\nu^3+(\nu+1)(1-\vert t\vert)+\vert t\vert+3}{\nu^3}\bigg)e^{-\vert s\vert-\vert t\vert}\\
&+\frac{1}{2}(s^2+\vert s\vert+1)\bigg(\frac{(\nu+1)(\vert t\vert(\nu+1)^2+3(\nu+1)-\vert t\vert)-1}{\nu^3}\\
&\ \ \ \ \ \ +(\nu+1)(t^2(\nu+2)(\nu+1)+2\vert t\vert(\nu+2)+t^2)\bigg)e^{-(\nu+1)\vert s\vert-\vert t\vert}\\
&-\frac{1}{2}s(\vert s\vert+1)(\nu+1)\bigg(\frac{t(\nu+2)^3(\nu+1)}{2\nu^2}+t(\nu+1)^2+t-4\sgn(t)\bigg)e^{-\vert s\vert-\vert t\vert}\\
&+\frac{1}{2}s(\vert s\vert+1)\bigg(\frac{t(\nu+2)^3(\vert t\vert(\nu+1)^3-3(\nu+1)^2+\vert t\vert(\nu+1)+1)}{4(\nu+1)\nu^2}\\
&\ \ \ \ \ \ -t(\nu+1)^2+t-4\sgn(t)(\nu+1)\bigg)e^{-\vert s\vert-\vert t\vert}\\
&+\frac{1}{2\pi}s(\vert s\vert+1)(\nu+2)^3\big(tJ_1(t)-2J_2(t)\big)e^{-\vert s\vert}\\
& \hspace*{-5mm} +\big((s^2+\vert s\vert+1)(t^2+\vert t\vert+1)+s(\vert s\vert+1)t(\vert t\vert+1)\big)\frac{(\nu+2)^2(5\nu^2+14\nu+10)}{64(\nu+1)^3}\cdot e^{-\vert s\vert-\vert t\vert},
\end{align*}
where $\sgn(\cdot)$ is the sign function, and 
\begin{equation*}
J_1(t;\nu)=\int_0^{\infty}\frac{x^2\cos(tx)}{(1+x^2)((\nu+1)^2+x^2)^2}\mbox{d}x\quad \mbox{and}\quad J_2(t;\nu)=\int_0^{\infty}\frac{x^3\sin(tx)}{(1+x^2)^2((\nu+1)^2+x^2)^2}\mbox{d}x.
\end{equation*}

\end{enumerate}

\end{remark}

As demonstrated in Remark \ref{rem:est}, the estimation procedure has some influence on the limit null distribution of $T_n$. We refer to \cite{C:2011,F:2013} for alternative estimators of the parameters of the Cauchy distribution and to \cite{CF:2006} for an Pitman estimator of the location parameter $\alpha$ when $\beta$ is known.

\subsection{Asymptotic normality of the limit statistic under the null hypothesis}
The next result gives the limiting distribution of the scaled limiting statistic $aT_{n,a}$ for $a\to 0$ given in (\ref{stat:lim1}).

\begin{theorem}\label{thmlimit2} 
\begin{enumerate}
\item[a)] 
Let $\arcsine(a,b)$ denote the arcsine distribution on $[a,b]$ with distribution function 
\[
F(x)=\frac{2}{\pi}\arcsin\sqrt{\frac{x-a}{b-a}},a\leq x\leq b.
\] 
If $X\sim C(0,1)$, then
\begin{align*}
\frac{4X^2}{(1+X^2)^2} &\sim \arcsine(0,1).
\end{align*}
\item[b)]
Let $X_1,...,X_n$ be i.i.d. $\mbox{C}(0,1)$ distributed random variables, and let $Y_{n,j}$ be the corresponding scaled residuals, using  any estimator $(\hat\alpha,\hat\beta)$ with linear representation. Then,
\begin{align} \label{stat:lim3}
T_{n,0} &= \sqrt{2n} \left( \frac{8}{n} \sum_{j=1}^n \frac{Y_{n,j}^2}{1+Y_{n,j}^2} - 1\right) \ \cd \ N(0,1),
\end{align}
where $N(0,1)$ denotes the standard normal distribution.
\end{enumerate}
\end{theorem}

\begin{proof}
\begin{enumerate}
\item[a)]
Let $U$ have a uniform distribution on $(-\pi/2,\pi/2)$, and put $X=\tan U$. Then, $X\sim C(0,1)$. Further, $\sin U, \, X^2/(1+X^2)=\sin^2 U, \, 1/(1+X^2)=\cos^2 U$ as well as $2\sin U \cos U$ have an $\arcsine(-1,1)$ distribution (see, e.g., \cite{N:1983}). If $Z \sim \arcsine(-1,1)$, then $Z^2 \sim \arcsine(0,1)$. Hence,
\begin{align*}
\frac{4X^2}{(1+X^2)^2} &\sim 4 \sin^2 U\cos^2 U \sim \arcsine(0,1).
\end{align*}
\item[b)]
A second order Taylor expansion around $(\alpha_0,\beta_0)=(0,1)$ yields
\begin{align*} 
\frac{1}{\sqrt{n}} & \sum_{j=1}^n \frac{(X_j-\alpha)^2/\beta^2}{(1+(X_j-\alpha)^2/\beta^2)^2}
= \frac{1}{\sqrt{n}} \sum_{j=1}^n \left\{ 
  \frac{X_j^2}{(1+X_j^2)^2)} 
+ \frac{2X_j(X_j^2-1)\alpha}{(1+X_j^2)^3} 
+ \frac{2X_j^2(X_j^2-1)(\beta-1)}{(1+X_j^2)^3}  \right. \\
& \left. + \frac{3X_j^4-8X_j^2+1)\alpha^2}{(1+X_j^2)^4} 
+ \frac{4X_j(X_j^4-4X_j^2+1)\alpha(\beta-1)}{(1+X_j^2)^4} 
+ \frac{X_j^2(X_j^4-8X_j^2+3)(\beta-1)^2}{(1+X_j^2)^4} \right\} + R_n,
\end{align*}
where $R_n\cp 0$ for $n\to\infty.$ If we replace $(\alpha,\beta)$ by any estimator $(\hat\alpha,\hat\beta)$ with linear representation, the second and third term converge to zero in probability, since, by the results given in the proof of part a),
\begin{align*} 
\frac{1}{n} \sum_{j=1}^n \frac{X_j(X_j^2-1)}{(1+X_j^2)^3} &\fse \E\frac{X_1(X_1^2-1)}{(1+X_1^2)^3} \ = \ 0, \\
\frac{1}{n} \sum_{j=1}^n \frac{X_j^2 (X_j^2-1)}{(1+X_j^2)^3} &\fse \E\frac{X_1^2(X_1^2-1)}{(1+X_1^2)^3} \ = \ 0,
\end{align*}
and $\sqrt{n}\hat\alpha$ and $\sqrt{n}(\hat\beta-1)$ have limiting normal distributions. Since the corresponding arithmetic means of the second order terms converge almost surely to finite values, the 4th, 5th and 6th term also converge to zero in probability. Therefore,
\begin{align*} 
\frac{8}{\sqrt{n}} \sum_{j=1}^n  \frac{Y_{n,j}^2}{(1+Y_{n,j}^2)^2}
&= \frac{8}{\sqrt{n}} \sum_{j=1}^n \frac{X_j^2}{(1+X_j^2)^2} + \tilde{R}_n,
\end{align*}
where $\tilde{R}_n\cp 0$ for $n\to\infty.$ By a), $8X_1^2/(1+X_1^2)^2$ has expected value 1 and variance 1/2, and the assertion follows by the central limit theorem.

\end{enumerate}
\end{proof}

\begin{remark} \label{remark:Tinfty}
Note that the sum appearing on the right hand side of (\ref{stat:lim2}) involves essentially the first component of the score function of the Cauchy distribution (see the second part of Remark \ref{rem:est}). Hence, when using the maximum-likelihood estimator for $(\alpha,\beta)$, $\lim_{a\to\infty} aT_{n,a}\equiv 0$. For other estimators, the limit does not vanish, but a goodness-of-fit test based on it has very low power. Hence, we don't give further results for this statistic.
\end{remark}

\section{Consistency and behaviour under fixed alternatives}
In this section we show that the new tests are consistent against all alternatives satisfying a weak moment condition. Let $X,X_1,X_2,\ldots$ be i.i.d. random variables with cumulative distribution function $F$ having a unique median as well as unique upper and lower quartiles and $\E |X|^3/(1+X^2)^2<\infty$. Since the tests $T_n$ are affine invariant, we assume w.l.o.g. that the true median $\alpha=0$ and the half interquartile range $\beta=1$. Under these assumptions we clearly have that 
\begin{equation*}
    (\widehat{\alpha}_n,\widehat{\beta}_n)\overset{\PP}{\longrightarrow}(0,1),\quad\mbox{as}\;n\rightarrow\infty,
\end{equation*}
for the MIQ estimation method in Remark \ref{rem:est}, for details see \cite{GH:2000}, and assume the convergence in all other cases.

\begin{theorem}\label{thm:consistency}
Under the standing assumptions, we have
\[\frac{T_n}{n}\overset{\PP}{\longrightarrow}\int_{-\infty}^{\infty}\bigg\vert\E\bigg[\bigg(it-\frac{2X}{1+X^2}\bigg)e^{itX}\bigg]\bigg\vert^2\omega(t)dt=\Delta_F,\quad\mbox{as}\;n\longrightarrow\infty.\]
\end{theorem}
A proof of Theorem \ref{thm:consistency} is deferred to Appendix \ref{app:cons}. In view of the characterization in Theorem \ref{thm:Char} $\Delta_F=0$ if and only if $F=\mbox{C}(0,1)$. Hence, we conclude that the tests $T_n$ are able to detect all alternatives satisfying the assumptions of this section.
\begin{remark}
In this remark we fix the weight function $\omega_1(t)=\exp(-|t|)$, $t\in\R$. By numerical integration we calculate that if $X\sim\mathcal{U}(-\sqrt{3},\sqrt{3})$, then $\Delta_{\mathcal{U}}\approx 2.332019$, if $X\sim\mbox{N}(0,1)$ then  $\Delta_{N}\approx0.021839$ and if $X\sim\mbox{Log}(0,1)$, then $\Delta_{L}\approx0.041495$. 
\end{remark}

In the following lines we use the theory presented in \cite{BEH:2017} to derive the limit distribution under fixed alternatives. Let $Z_n^{\bullet}(t)=n^{-1/2}Z_n(t)$, and
\[z(t)=\E\big[\big(\frac{2X}{1+X^2}+t\big)\cos(tX)+\big(t-\frac{2X}{1+X^2}\big)\sin(tX)\big],\quad t\in\R,\]
as well as $W_n(t)=\sqrt{n}(Z_n^{\bullet}(t)-z(t))$, $t\in\R.$ The process $W_n=(W_n(t),t\in\R)$ is a random element of $\mathbb{H}$ and we have
\begin{align}
\label{Tn-Delta}
\sqrt{n}\Big(\frac{T_n}{n}-\Delta\Big)&=\sqrt{n}(\Vert Z_n^{\bullet}\Vert_{\mathbb{H}}^2-\Vert z\Vert_{\mathbb{H}}^2)\nonumber=\sqrt{n}\langle Z_n^{\bullet}-z,Z_n^{\bullet}+z\rangle_{\mathbb{H}}\nonumber=\sqrt{n}\langle Z_n^{\bullet}-z,2z+Z_n^{\bullet}-z\rangle_{\mathbb{H}}\nonumber\\
&=2\langle\sqrt{n}(Z_n^{\bullet}-z),z\rangle_{\mathbb{H}}+\frac{1}{\sqrt{n}}\Vert\sqrt{n}(Z_n^{\bullet}-z)\Vert_{\mathbb{H}}^2=2\langle W_n,z\rangle_{\mathbb{H}}+\frac{1}{\sqrt{n}}\Vert W_n\Vert_\mathbb{H}^2.
\end{align}
The next lemma is needed to prove the subsequent statements. For a proof see Appendix \ref{app:PL}.
\begin{lemma}\label{lem:Hilfslemma}
Under the standing assumptions, we have for $n\longrightarrow\infty$
\[W_n=\sqrt{n}(Z_n^{\bullet}-z)\overset{\mathcal{D}}{\longrightarrow}\mathcal{W}\]
in $\mathbb{H}$, where $\mathcal{W}$ is a centred Gaussian process in $\mathbb{H}$ with covariance kernel
\begin{align*}
K(s,t)=&\E\bigg[\bigg(st+\frac{4X^2}{(1+X^2)^2}\bigg)\cos(X(s-t))+\bigg(st-\frac{4X^2}{(1+X^2)^2}\bigg)\sin(X(s+t))\bigg]\\
&+\E\bigg[(s+t)\frac{2X}{1+X^2}\cos(X(s+t))+(s-t)\frac{2X}{1+X^2}\sin(X(s-t))\bigg]\\
&+\E\bigg[\bigg(\frac{2X}{1+X^2}+s\bigg)(a(t)\psi_1(X)+b(t)\psi_2(X))\cos(sX)\bigg]\\
&+\E\bigg[\bigg(\frac{2X}{1+X^2}+t\bigg)(a(s)\psi_1(X)+b(s)\psi_2(X))\cos(tX)\bigg]\\
&+\E\bigg[\bigg(s-\frac{2X}{1+X^2}\bigg)(a(t)\psi_1(X)+b(t)\psi_2(X))\sin(sX)\bigg]\\
&+\E\bigg[\bigg(t-\frac{2X}{1+X^2}\bigg)(a(s)\psi_1(X)+b(s)\psi_2(X))\sin(tX)\bigg]\\
&+\E[\psi_1(X)^2]a(s)a(t)+\E[\psi_2(X)^2]b(s)b(t)-3z(s)z(t).
\end{align*}
Here, $\psi_i(\cdot),\ i=1,2,$ are defined as in (\ref{eq:psi_11}), $a(t)=\E[g(t,X)],$ $b(t)=E[Xg(t,X)]$, and $g(\cdot,\cdot)$ is given in \eqref{g}.
\end{lemma}
 By Lemma \ref{lem:Hilfslemma} we conclude that $\sqrt{n}(Z_n^{\bullet}-z)$ is a tight sequence in $\mathbb{H}$, hence $\frac{1}{\sqrt{n}}\Vert\sqrt{n}(Z_n^{\bullet}-z)\Vert_{\mathbb{H}}^2=o_\PP(1)$, and we have, by Slutzky's theorem, that the weak limit of $\sqrt{n}(\frac{Z_n}{n}-\Delta)$ in (\ref{Tn-Delta}) only depends on $2\langle\sqrt{n}(Z_n^{\bullet}-z),z\rangle_{\mathbb{H}}$. Theorem 1 in \cite{BEH:2017} and Lemma \ref{lem:Hilfslemma} then directly prove the following theorem.

\begin{theorem}\label{thm:fA}
Under the standing assumptions, we have
\[\sqrt{n}\bigg(\frac{T_n}{n}-\Delta_F\bigg)\overset{\mathcal{D}}{\longrightarrow}\mbox{N}(0,\sigma^2),\quad\mbox{as}\;n\rightarrow\infty,\]
where
\begin{equation}
\label{sigma}
\sigma^2=4\int_{-\infty}^{\infty}\int_{-\infty}^{\infty}K(s,t)z(s)z(t)\omega(s)\omega(t)\,\mbox{d}s\,\mbox{d}t.
\end{equation}
\end{theorem}

The results of Theorem \ref{thm:fA} can be used to derive confidence intervals for $\Delta_F$ or approximations of the power function. Note that since for most fixed alternatives and estimation procedures it is hard to explicitly derive the formulae for $K$, $z$ and $\sigma^2$, we propose to estimate $\sigma^2$ by a consistent estimator $\widehat{\sigma}_n^2$, for a construction of such an estimator and for further details on the applications of Theorem \ref{thm:fA}, we refer to \cite{BEH:2017}.

\section{Simulation results}

In this section, we describe an extensive simulation study to compare the power of the new tests with established ones.
All simulations have been done for nominal level $\alpha=0.05$ and for sample sizes $n=20$ and $n=50$. 
The unknown parameters are estimated either by median and IQR estimators or by the method of maximum likelihood, see Remark \ref{rem:est}. In a first step, critical values are determined for each test statistic by a Monte Carlo simulation with $10^5$ replications. In a second step, empirical power of the different tests is calculated based on $10^4$ replications.

Besides the new tests based on $T_{n,a}$ in (\ref{eq:EF}) with weights $a=1,2,3,4,5,6$, we use the limiting test $T_{n,0}$ in (\ref{stat:lim3}).  Power decreased for larger weights for all distributions used in the study; hence, we omit the results for weights larger than 6 as well as for the test based on the limiting statistic $T_{n,\infty}$ in (\ref{stat:lim2}), see also Remark \ref{remark:Tinfty}. Moreover, we use the test based on the Kullback-Leibler distance ($KL$) described in \cite{MZ:2017}. For the nonparametric estimation of the entropy, the authors recommend the window size $m=4$ for $n=20$ and $m=20$ for $n=50$ (\cite{MZ:2017}, p. 1109), and we followed this suggestion. 
From the classical tests that utilize the empirical distribution function (so-called edf tests), we choose the following: Kolmogorov-Smirnov test ($KS$), Cram\'{e}r-von Mises test ($CM$, Anderson Darling test ($AD$ and Watson test ($W$). Finally, we apply the tests based on the statistics $D_{n,\lambda}$ considered in \cite{GH:2000} and \cite{MT:2005}, with weights $\lambda=1,2,3,4,5,6$.

Besides the standard Cauchy and normal distribution (C$(0,1)$ and N$(0,1)$, for short), we use Cauchy-normal mixtures CN$(p)=(1-p)$C$(0,1)+p$N$(0,1)$ with $p=0.5$ and $p=0.8$. Further, we employ Student's $t$-distribution (Student$(k)$) with $k=2,3,5,10$ degrees of freedom, the (symmetric) Tukey g-h distribution  with $g=0$  and $h=0.2, 0.1, 0.05$, denoted by Tukey($h$), and the Tukey lambda distribution (Tukey-L$(\lambda)$) with $\lambda=-3,-2,-0.5,0.5$. Note that the  Tukey lambda distribution approximates the Cauchy distribution for $\lambda=-1$, it coincides with the logistic distribution for $\lambda=0$, it is close to the normal for $\lambda=0.14$, and coincides with the uniform distribution for $\lambda=1$.
In the family of stable distributions S$(\alpha,\beta)$, we choose symmetric distributions S$(\alpha,0)$ with $\alpha=0.4,0.7,1.2,1.5,1.8$, and skewed distributions S$(\alpha,1)$ with $\alpha=0.5,1,1.5,2$. Moreover, we use the uniform, logistic, Laplace, Gumbel and exponential distribution, and the Mittag-Leffler distribution ML$(\alpha)$ with $\alpha=0.25,0.5,0.75$. The ML$(\alpha)$ distribution has Laplace transform $(1+t^\alpha)^{-1}$, and coincides with the exponential distribution for $\alpha=1$.

Power estimates of the tests under discussion are given in Tables \ref{Table1}-\ref{Table4}. All entries are the percentage of rejection of $H_0$, rounded to the nearest integer. In Tables \ref{Table1} and \ref{Table2}, the results using the median and interquantile-distance estimators are given, with sample size $n=20$ and $n=50$, respectively. Tables \ref{Table3} and \ref{Table4} show the corresponding results for the maximum likelihood estimator.

\begin{sidewaystable}
\begin{center}
\begin{tabular}{rrrrrrrrrrrrrrrrrrr} 
\hline
 & $T_{n,1}$ & $T_{n,2}$ & $T_{n,3}$ & $T_{n,4}$ & $T_{n,5}$ & $T_{n,6}$ & $T_{n,0}$ 
 & $KL$ & $KS$ & $CM$ & $AD$ & $W$ & $D_{n,1}$ & $D_{n,2}$ & $D_{n,3}$ & $D_{n,4}$ & $D_{n,5}$ & $D_{n,6}$ \\ 
\hline
C(0,1) & 5 & 5 & 5 & 5 & 5 & 5 & 5 & 5 & 5 & 5 & 5 & 5 & 5 & 5 & 5 & 5 & 5 & 5 \\
  N(0,1) & 8 & 4 & 1 & 0 & 0 & 0 & 29 & 73 & 5 & 6 & 6 & 15 & 15 & 21 & 26 & 26 & 23 & 19 \\
  CN(0.5) & 4 & 3 & 3 & 3 & 3 & 3 & 9 & 14 & 4 & 3 & 3 & 5 & 5 & 5 & 5 & 5 & 4 & 3 \\
  CN(0.8) & 6 & 3 & 1 & 1 & 1 & 1 & 19 & 36 & 4 & 4 & 4 & 8 & 8 & 11 & 13 & 12 & 10 & 8 \\
  Student(2) & 4 & 2 & 2 & 2 & 2 & 2 & 9 & 22 & 3 & 4 & 3 & 5 & 4 & 5 & 5 & 5 & 5 & 3 \\
  Student(3) & 5 & 2 & 1 & 1 & 1 & 1 & 14 & 34 & 3 & 4 & 3 & 7 & 6 & 7 & 9 & 9 & 8 & 6 \\
  Student(5) & 6 & 3 & 1 & 1 & 1 & 1 & 20 & 49 & 4 & 4 & 4 & 9 & 8 & 11 & 14 & 14 & 12 & 9 \\
  Student(10) & 7 & 3 & 1 & 1 & 1 & 1 & 24 & 61 & 4 & 5 & 5 & 11 & 11 & 15 & 18 & 19 & 16 & 13 \\
  Stable(0.4,0) & 55 & 52 & 39 & 26 & 18 & 13 & 79 & 2 & 39 & 45 & 83 & 58 & 62 & 64 & 66 & 69 & 71 & 74 \\
  Stable(0.7,0) & 16 & 15 & 12 & 9 & 8 & 7 & 21 & 2 & 12 & 13 & 23 & 14 & 16 & 18 & 20 & 22 & 24 & 25 \\
  Stable(1.2,0) & 4 & 3 & 3 & 3 & 3 & 3 & 6 & 10 & 4 & 4 & 3 & 4 & 4 & 4 & 4 & 4 & 3 & 3 \\
  Stable(1.5,0) & 5 & 2 & 1 & 1 & 1 & 2 & 14 & 28 & 4 & 4 & 3 & 6 & 6 & 7 & 8 & 8 & 6 & 5 \\
  Stable(1.8,0) & 7 & 3 & 1 & 1 & 1 & 1 & 22 & 52 & 5 & 5 & 5 & 11 & 10 & 14 & 16 & 16 & 14 & 11 \\
  Stable(0.5,1) & 91 & 97 & 90 & 76 & 61 & 50 & 40 & 65 & 94 & 90 & 97 & 88 & 95 & 76 & 57 & 46 & 44 & 46 \\
  Stable(1.5,1) & 10 & 10 & 7 & 6 & 6 & 6 & 19 & 49 & 9 & 9 & 9 & 11 & 13 & 11 & 12 & 11 & 10 & 8 \\
  Stable(2,1) & 8 & 4 & 0 & 0 & 0 & 0 & 28 & 73 & 5 & 6 & 6 & 14 & 14 & 20 & 24 & 24 & 22 & 17 \\
  Tukey(0.2) & 5 & 2 & 1 & 1 & 1 & 1 & 14 & 36 & 3 & 4 & 3 & 7 & 6 & 8 & 10 & 9 & 8 & 7 \\
  Tukey(0.1) & 5 & 3 & 1 & 1 & 1 & 1 & 21 & 52 & 4 & 4 & 4 & 9 & 8 & 12 & 15 & 15 & 13 & 10 \\
  Tukey(0.05) & 7 & 3 & 1 & 0 & 0 & 0 & 24 & 62 & 4 & 5 & 5 & 12 & 11 & 15 & 19 & 19 & 17 & 13 \\
  Tukey-L(-3) & 43 & 42 & 33 & 24 & 17 & 12 & 64 & 2 & 31 & 35 & 79 & 44 & 52 & 58 & 62 & 65 & 68 & 71 \\
  Tukey-L(-2) & 21 & 22 & 18 & 14 & 11 & 10 & 33 & 1 & 16 & 17 & 41 & 19 & 24 & 28 & 32 & 35 & 38 & 41 \\
  Tukey-L(-0.5) & 4 & 3 & 2 & 2 & 2 & 2 & 7 & 15 & 3 & 4 & 3 & 5 & 4 & 4 & 4 & 4 & 3 & 3 \\
  Tukey-L(0.5) & 17 & 11 & 1 & 0 & 0 & 0 & 45 & 94 & 10 & 12 & 15 & 31 & 36 & 45 & 50 & 50 & 47 & 40 \\
  Uniform & 32 & 21 & 1 & 1 & 0 & 0 & 56 & 99 & 24 & 22 & 28 & 51 & 60 & 67 & 69 & 68 & 65 & 59 \\
  Logistic & 6 & 3 & 1 & 1 & 1 & 1 & 21 & 56 & 4 & 5 & 4 & 11 & 9 & 13 & 16 & 17 & 15 & 11 \\
  Laplace & 4 & 2 & 1 & 1 & 1 & 1 & 10 & 32 & 3 & 3 & 3 & 6 & 5 & 6 & 8 & 8 & 7 & 5 \\
  Gumbel & 10 & 9 & 3 & 3 & 2 & 2 & 26 & 70 & 9 & 8 & 8 & 15 & 17 & 18 & 20 & 20 & 18 & 15 \\
  ML(0.25) & 100 & 100 & 100 & 97 & 92 & 83 & 93 & 89 & 100 & 100 & 100 & 100 & 100 & 99 & 98 & 96 & 96 & 96 \\
  ML(0.5) & 98 & 99 & 93 & 79 & 63 & 49 & 55 & 81 & 99 & 96 & 99 & 97 & 99 & 86 & 69 & 58 & 55 & 57 \\
  ML(0.75) & 78 & 87 & 69 & 52 & 43 & 39 & 20 & 75 & 85 & 71 & 79 & 72 & 83 & 52 & 33 & 25 & 23 & 23 \\
  Exponential & 39 & 44 & 23 & 15 & 13 & 11 & 27 & 92 & 48 & 31 & 32 & 41 & 49 & 34 & 30 & 28 & 27 & 25 \\  
 \hline
\end{tabular}
\caption{\label{Table1} Percentage of 10 000 MC samples declared significant by various
tests for the Cauchy distribution using median and IQR estimator ($\alpha=0.05, n=20$)}
\end{center}
\end{sidewaystable}

\begin{sidewaystable}
\begin{center}
\begin{tabular}{rrrrrrrrrrrrrrrrrrr}
  \hline
 & $T_{n,1}$ & $T_{n,2}$ & $T_{n,3}$ & $T_{n,4}$ & $T_{n,5}$ & $T_{n,6}$ & $T_{n,0}$ 
 & $KL$ & $KS$ & $CM$ & $AD$ & $W$ & $D_{n,1}$ & $D_{n,2}$ & $D_{n,3}$ & $D_{n,4}$ & $D_{n,5}$ & $D_{n,6}$ \\ 
  \hline
C(0,1) & 5 & 5 & 4 & 5 & 5 & 5 & 5 & 5 & 5 & 5 & 5 & 5 & 5 & 5 & 5 & 5 & 5 & 5 \\
  N(0,1) & 24 & 24 & 4 & 1 & 0 & 0 & 72 & 100 & 26 & 29 & 55 & 67 & 65 & 83 & 89 & 92 & 93 & 94 \\
  CN(0.5) & 8 & 5 & 3 & 3 & 3 & 3 & 23 & 14 & 7 & 6 & 7 & 13 & 13 & 15 & 16 & 16 & 15 & 13 \\
  CN(0.8) & 16 & 13 & 2 & 1 & 1 & 1 & 51 & 46 & 14 & 14 & 23 & 38 & 37 & 49 & 54 & 55 & 55 & 54 \\
  Student(2) & 6 & 4 & 2 & 2 & 2 & 2 & 23 & 52 & 6 & 6 & 8 & 14 & 10 & 15 & 19 & 21 & 23 & 23 \\
  Student(3) & 10 & 7 & 2 & 2 & 1 & 1 & 38 & 82 & 9 & 10 & 16 & 27 & 21 & 32 & 40 & 44 & 47 & 49 \\
  Student(5) & 14 & 12 & 2 & 1 & 1 & 1 & 52 & 96 & 13 & 15 & 28 & 41 & 36 & 52 & 62 & 67 & 70 & 72 \\
  Student(10) & 18 & 18 & 3 & 1 & 1 & 0 & 64 & 100 & 18 & 20 & 41 & 54 & 49 & 70 & 78 & 83 & 85 & 86 \\
  Stable(0.4,0) & 91 & 90 & 76 & 47 & 25 & 15 & 99 & 0 & 80 & 91 & 99 & 98 & 97 & 96 & 97 & 97 & 98 & 98 \\
  Stable(0.7,0) & 23 & 23 & 14 & 10 & 8 & 7 & 46 & 0 & 16 & 17 & 39 & 26 & 29 & 34 & 37 & 41 & 45 & 48 \\
  Stable(1.2,0) & 6 & 4 & 3 & 3 & 3 & 3 & 12 & 19 & 6 & 5 & 5 & 8 & 6 & 7 & 8 & 8 & 8 & 8 \\
  Stable(1.5,0) & 11 & 8 & 3 & 2 & 2 & 2 & 38 & 53 & 9 & 9 & 13 & 24 & 20 & 29 & 34 & 36 & 37 & 37 \\
  Stable(1.8,0) & 18 & 16 & 3 & 1 & 1 & 1 & 62 & 88 & 18 & 19 & 35 & 50 & 46 & 63 & 71 & 74 & 77 & 77 \\
  Stable(0.5,1) & 100 & 100 & 100 & 100 & 97 & 87 & 71 & 13 & 100 & 100 & 100 & 100 & 100 & 100 & 100 & 100 & 99 & 99 \\
  Stable(1.5,1) & 28 & 41 & 26 & 18 & 15 & 13 & 52 & 81 & 56 & 38 & 53 & 60 & 61 & 71 & 72 & 70 & 69 & 68 \\
  Stable(2,1) & 23 & 23 & 4 & 1 & 1 & 0 & 72 & 100 & 26 & 28 & 55 & 67 & 65 & 84 & 89 & 92 & 94 & 94 \\
  Tukey(0.2) & 10 & 7 & 2 & 1 & 1 & 1 & 40 & 87 & 9 & 9 & 16 & 27 & 22 & 34 & 42 & 47 & 50 & 52 \\
  Tukey(0.1) & 15 & 13 & 2 & 1 & 1 & 1 & 56 & 98 & 15 & 16 & 32 & 44 & 39 & 57 & 67 & 72 & 75 & 77 \\
  Tukey(0.05) & 19 & 17 & 3 & 1 & 1 & 1 & 64 & 100 & 18 & 21 & 41 & 54 & 50 & 70 & 78 & 83 & 85 & 87 \\
  Tukey-L(-3) & 70 & 74 & 57 & 35 & 20 & 13 & 96 & 0 & 64 & 73 & 99 & 91 & 90 & 94 & 95 & 96 & 97 & 98 \\
  Tukey-L(-2) & 33 & 35 & 24 & 16 & 12 & 10 & 64 & 0 & 26 & 29 & 72 & 47 & 48 & 58 & 64 & 68 & 72 & 75 \\
  Tukey-L(-0.5) & 5 & 4 & 3 & 3 & 3 & 3 & 13 & 37 & 5 & 5 & 5 & 8 & 7 & 8 & 10 & 11 & 12 & 12 \\
  Tukey-L(0.5) & 56 & 61 & 11 & 1 & 0 & 0 & 90 & 100 & 68 & 63 & 89 & 92 & 96 & 99 & 100 & 100 & 100 & 100 \\
  Uniform & 87 & 87 & 18 & 1 & 0 & 0 & 97 & 100 & 95 & 86 & 98 & 99 & 100 & 100 & 100 & 100 & 100 & 100 \\
  Logistic & 15 & 14 & 3 & 1 & 1 & 1 & 59 & 100 & 15 & 17 & 35 & 48 & 43 & 63 & 72 & 77 & 80 & 82 \\
  Laplace & 6 & 4 & 2 & 1 & 1 & 1 & 21 & 94 & 6 & 7 & 12 & 18 & 13 & 24 & 32 & 37 & 40 & 42 \\
  Gumbel & 33 & 43 & 18 & 8 & 6 & 5 & 68 & 100 & 58 & 42 & 64 & 73 & 74 & 86 & 89 & 90 & 91 & 91 \\
  ML(0.25) & 100 & 100 & 100 & 100 & 100 & 100 & 100 & 4 & 100 & 100 & 100 & 100 & 100 & 100 & 100 & 100 & 100 & 100 \\
  ML(0.5) & 100 & 100 & 100 & 100 & 98 & 91 & 86 & 15 & 100 & 100 & 100 & 100 & 100 & 100 & 100 & 100 & 100 & 100 \\
  ML(0.75) & 100 & 100 & 100 & 93 & 78 & 62 & 30 & 45 & 100 & 100 & 100 & 100 & 100 & 100 & 100 & 98 & 94 & 91 \\
  Exponential & 94 & 96 & 75 & 44 & 32 & 26 & 64 & 100 & 100 & 92 & 98 & 99 & 100 & 100 & 99 & 98 & 97 & 97 \\
   \hline
\end{tabular}
\caption{\label{Table2} Percentage of 10 000 MC samples declared significant by various
tests for the Cauchy distribution using median and IQR estimator ($\alpha=0.05, n=50$)}
\end{center}
\end{sidewaystable}

\begin{sidewaystable}
\begin{center}
\begin{tabular}{rrrrrrrrrrrrrrrrrrr}
  \hline
 & $T_{n,1}$ & $T_{n,2}$ & $T_{n,3}$ & $T_{n,4}$ & $T_{n,5}$ & $T_{n,6}$ & $T_{n,0}$ 
 & $KL$ & $KS$ & $CM$ & $AD$ & $W$ & $D_{n,1}$ & $D_{n,2}$ & $D_{n,3}$ & $D_{n,4}$ & $D_{n,5}$ & $D_{n,6}$ \\ 
  \hline
C(0,1) & 5 & 5 & 5 & 5 & 5 & 5 & 5 & 5 & 5 & 5 & 5 & 5 & 5 & 5 & 5 & 5 & 5 & 5 \\
  N(0,1) & 16 & 30 & 34 & 28 & 10 & 3 & 26 & 76 & 5 & 3 & 2 & 29 & 23 & 17 & 10 & 7 & 4 & 2 \\
  CN(0.5) & 8 & 10 & 8 & 5 & 2 & 2 & 9 & 14 & 4 & 4 & 3 & 9 & 7 & 3 & 2 & 2 & 2 & 1 \\
  CN(0.8) & 13 & 19 & 18 & 12 & 5 & 2 & 18 & 37 & 5 & 3 & 2 & 18 & 14 & 8 & 5 & 3 & 2 & 1 \\
  Student(2) & 8 & 9 & 8 & 6 & 3 & 2 & 10 & 22 & 4 & 3 & 2 & 10 & 7 & 4 & 2 & 2 & 1 & 1 \\
  Student(3) & 10 & 14 & 14 & 10 & 4 & 2 & 15 & 37 & 4 & 2 & 1 & 14 & 10 & 6 & 3 & 2 & 2 & 1 \\
  Student(5) & 12 & 19 & 20 & 14 & 6 & 2 & 19 & 51 & 5 & 3 & 2 & 19 & 14 & 9 & 5 & 4 & 3 & 2 \\
  Student(10) & 15 & 25 & 27 & 20 & 8 & 3 & 22 & 64 & 5 & 3 & 2 & 24 & 19 & 13 & 8 & 5 & 3 & 2 \\
  Stable(0.4,0) & 53 & 64 & 72 & 77 & 80 & 81 & 76 & 0 & 42 & 33 & 77 & 77 & 71 & 79 & 83 & 86 & 87 & 88 \\
  Stable(0.7,0) & 10 & 14 & 19 & 22 & 25 & 26 & 19 & 1 & 12 & 11 & 18 & 15 & 18 & 24 & 26 & 28 & 29 & 30 \\
  Stable(1.2,0) & 7 & 6 & 5 & 4 & 2 & 2 & 7 & 11 & 4 & 4 & 3 & 7 & 5 & 3 & 2 & 2 & 2 & 2 \\
  Stable(1.5,0) & 10 & 13 & 12 & 8 & 3 & 2 & 14 & 29 & 4 & 3 & 2 & 14 & 9 & 5 & 3 & 2 & 2 & 1 \\
  Stable(1.8,0) & 14 & 23 & 25 & 18 & 7 & 3 & 21 & 55 & 5 & 3 & 2 & 23 & 17 & 11 & 7 & 5 & 3 & 2 \\
  Stable(0.5,1) & 55 & 70 & 78 & 81 & 83 & 84 & 26 & 54 & 99 & 99 & 99 & 82 & 79 & 82 & 85 & 87 & 88 & 89 \\
  Stable(1.5,1) & 14 & 21 & 22 & 17 & 11 & 7 & 16 & 50 & 14 & 12 & 9 & 22 & 17 & 11 & 9 & 7 & 6 & 5 \\
  Stable(2,1) & 17 & 29 & 34 & 28 & 10 & 4 & 26 & 76 & 5 & 3 & 2 & 28 & 23 & 17 & 10 & 7 & 4 & 3 \\
  Tukey(0.2) & 10 & 14 & 14 & 10 & 4 & 2 & 14 & 38 & 4 & 3 & 2 & 14 & 10 & 6 & 4 & 3 & 2 & 1 \\
  Tukey(0.1) & 13 & 20 & 21 & 16 & 6 & 2 & 19 & 55 & 5 & 3 & 2 & 20 & 15 & 10 & 5 & 4 & 2 & 2 \\
  Tukey(0.05) & 15 & 25 & 28 & 21 & 7 & 3 & 23 & 65 & 5 & 3 & 2 & 24 & 19 & 13 & 8 & 5 & 3 & 2 \\
  Tukey-L(-3) & 34 & 49 & 60 & 67 & 72 & 75 & 60 & 0 & 35 & 27 & 72 & 60 & 58 & 71 & 76 & 79 & 81 & 82 \\
  Tukey-L(-2) & 13 & 21 & 29 & 35 & 40 & 42 & 29 & 1 & 17 & 15 & 34 & 26 & 27 & 38 & 42 & 45 & 48 & 49 \\
  Tukey-L(-0.5) & 6 & 7 & 6 & 4 & 3 & 2 & 8 & 16 & 4 & 3 & 2 & 7 & 5 & 3 & 2 & 2 & 1 & 1 \\
  Tukey-L(0.5) & 31 & 55 & 64 & 59 & 24 & 8 & 37 & 96 & 11 & 6 & 5 & 50 & 50 & 43 & 28 & 19 & 11 & 7 \\
  Uniform & 48 & 74 & 81 & 78 & 43 & 14 & 45 & 99 & 22 & 13 & 10 & 66 & 72 & 65 & 48 & 35 & 24 & 14 \\
  Logistic & 14 & 22 & 24 & 18 & 7 & 3 & 21 & 60 & 5 & 3 & 2 & 22 & 17 & 11 & 6 & 4 & 3 & 2 \\
  Laplace & 8 & 10 & 10 & 8 & 4 & 2 & 10 & 35 & 3 & 2 & 1 & 10 & 8 & 5 & 3 & 2 & 2 & 1 \\
  Gumbel & 18 & 31 & 35 & 30 & 16 & 9 & 24 & 74 & 13 & 10 & 7 & 31 & 25 & 19 & 14 & 12 & 9 & 7 \\
  ML(0.25) & 99 & 99 & 100 & 100 & 100 & 100 & 92 & 77 & 100 & 100 & 100 & 100 & 100 & 100 & 100 & 100 & 100 & 100 \\
  ML(0.5) & 80 & 86 & 90 & 91 & 92 & 92 & 43 & 72 & 100 & 100 & 100 & 96 & 93 & 92 & 94 & 95 & 95 & 96 \\
  ML(0.75) & 52 & 63 & 66 & 66 & 65 & 64 & 12 & 70 & 96 & 93 & 91 & 72 & 66 & 62 & 64 & 66 & 67 & 68 \\
  Exponential & 38 & 53 & 57 & 54 & 45 & 36 & 20 & 92 & 64 & 50 & 43 & 55 & 50 & 41 & 39 & 38 & 37 & 35 \\
   \hline
\end{tabular}
\caption{\label{Table3} Percentage of 10 000 MC samples declared significant by various
tests for the Cauchy distribution using ML estimation ($\alpha=0.05, n=20$)}
\end{center}
\end{sidewaystable}

\begin{sidewaystable}
\begin{center}
\begin{tabular}{rrrrrrrrrrrrrrrrrrr}
  \hline
 & $T_{n,1}$ & $T_{n,2}$ & $T_{n,3}$ & $T_{n,4}$ & $T_{n,5}$ & $T_{n,6}$ & $T_{n,0}$ 
 & $KL$ & $KS$ & $CM$ & $AD$ & $W$ & $D_{n,1}$ & $D_{n,2}$ & $D_{n,3}$ & $D_{n,4}$ & $D_{n,5}$ & $D_{n,6}$ \\ 
  \hline
C(0,1) & 5 & 5 & 5 & 5 & 5 & 5 & 5 & 5 & 5 & 5 & 5 & 5 & 5 & 5 & 5 & 5 & 5 & 5 \\
  N(0,1) & 40 & 76 & 90 & 95 & 96 & 96 & 64 & 100 & 16 & 9 & 15 & 77 & 77 & 87 & 90 & 91 & 91 & 91 \\
  CN(0.5) & 14 & 20 & 20 & 15 & 10 & 6 & 22 & 14 & 6 & 4 & 3 & 19 & 16 & 13 & 10 & 9 & 8 & 7 \\
  CN(0.8) & 28 & 53 & 61 & 60 & 52 & 39 & 48 & 46 & 11 & 6 & 6 & 50 & 48 & 50 & 48 & 46 & 43 & 39 \\
  Student(2) & 12 & 20 & 24 & 23 & 20 & 15 & 23 & 52 & 5 & 3 & 3 & 21 & 16 & 15 & 14 & 15 & 14 & 13 \\
  Student(3) & 19 & 36 & 45 & 48 & 46 & 38 & 38 & 82 & 7 & 4 & 4 & 38 & 30 & 33 & 34 & 36 & 35 & 35 \\
  Student(5) & 26 & 51 & 66 & 71 & 71 & 66 & 48 & 96 & 9 & 5 & 6 & 54 & 47 & 55 & 58 & 60 & 61 & 60 \\
  Student(10) & 32 & 65 & 81 & 86 & 88 & 86 & 57 & 100 & 13 & 6 & 10 & 67 & 63 & 74 & 78 & 80 & 80 & 80 \\
  Stable(0.4,0) & 95 & 98 & 99 & 99 & 99 & 99 & 99 & 0 & 82 & 81 & 99 & 100 & 99 & 99 & 100 & 100 & 100 & 100 \\
  Stable(0.7,0) & 23 & 34 & 40 & 45 & 47 & 49 & 44 & 0 & 18 & 15 & 31 & 38 & 36 & 44 & 49 & 52 & 53 & 55 \\
  Stable(1.2,0) & 8 & 10 & 10 & 8 & 6 & 4 & 12 & 18 & 5 & 4 & 3 & 10 & 8 & 6 & 5 & 5 & 4 & 4 \\
  Stable(1.5,0) & 19 & 34 & 40 & 39 & 33 & 25 & 35 & 53 & 7 & 4 & 4 & 34 & 28 & 28 & 27 & 27 & 25 & 24 \\
  Stable(1.8,0) & 34 & 62 & 76 & 80 & 78 & 71 & 56 & 88 & 12 & 6 & 9 & 64 & 60 & 67 & 68 & 69 & 68 & 67 \\
  Stable(0.5,1) & 98 & 100 & 100 & 100 & 100 & 100 & 46 & 6 & 100 & 100 & 100 & 100 & 100 & 100 & 100 & 100 & 100 & 100 \\
  Stable(1.5,1) & 32 & 61 & 74 & 78 & 76 & 68 & 44 & 80 & 57 & 43 & 46 & 62 & 60 & 66 & 66 & 67 & 67 & 66 \\
  Stable(2,1) & 41 & 77 & 90 & 95 & 97 & 96 & 65 & 100 & 17 & 9 & 16 & 78 & 77 & 88 & 90 & 92 & 92 & 91 \\
  Tukey(0.2) & 20 & 36 & 46 & 49 & 48 & 42 & 37 & 87 & 7 & 3 & 4 & 39 & 31 & 35 & 36 & 38 & 38 & 38 \\
  Tukey(0.1) & 27 & 55 & 70 & 76 & 77 & 73 & 51 & 98 & 9 & 5 & 7 & 57 & 52 & 61 & 65 & 67 & 67 & 67 \\
  Tukey(0.05) & 33 & 66 & 81 & 88 & 89 & 87 & 57 & 100 & 13 & 6 & 10 & 68 & 65 & 75 & 79 & 81 & 82 & 82 \\
  Tukey-L(-3) & 79 & 92 & 96 & 97 & 98 & 98 & 96 & 0 & 68 & 60 & 98 & 96 & 94 & 97 & 98 & 99 & 99 & 99 \\
  Tukey-L(-2) & 33 & 51 & 63 & 69 & 73 & 76 & 63 & 0 & 29 & 23 & 62 & 60 & 57 & 69 & 75 & 78 & 80 & 81 \\
  Tukey-L(-0.5) & 8 & 11 & 13 & 12 & 11 & 8 & 13 & 37 & 4 & 3 & 3 & 12 & 9 & 8 & 8 & 7 & 7 & 7 \\
  Tukey-L(0.5) & 75 & 98 & 100 & 100 & 100 & 100 & 83 & 100 & 54 & 27 & 50 & 97 & 99 & 100 & 100 & 100 & 100 & 100 \\
  Uniform & 95 & 100 & 100 & 100 & 100 & 100 & 91 & 100 & 91 & 53 & 80 & 100 & 100 & 100 & 100 & 100 & 100 & 100 \\
  Logistic & 28 & 58 & 74 & 81 & 82 & 80 & 52 & 100 & 10 & 5 & 8 & 60 & 56 & 66 & 70 & 72 & 73 & 72 \\
  Laplace & 10 & 21 & 30 & 36 & 38 & 36 & 19 & 94 & 4 & 2 & 2 & 23 & 20 & 25 & 25 & 26 & 26 & 26 \\
  Gumbel & 44 & 79 & 91 & 94 & 96 & 94 & 59 & 100 & 57 & 36 & 43 & 80 & 80 & 88 & 90 & 91 & 91 & 91 \\
  ML(0.25) & 100 & 100 & 100 & 100 & 100 & 100 & 100 & 0 & 100 & 100 & 100 & 100 & 100 & 100 & 100 & 100 & 100 & 100 \\
  ML(0.5) & 100 & 100 & 100 & 100 & 100 & 100 & 76 & 6 & 100 & 100 & 100 & 100 & 100 & 100 & 100 & 100 & 100 & 100 \\
  ML(0.75) & 97 & 99 & 100 & 100 & 100 & 99 & 14 & 38 & 100 & 100 & 100 & 100 & 100 & 100 & 100 & 100 & 100 & 100 \\
  Exponential & 87 & 98 & 99 & 100 & 100 & 99 & 48 & 100 & 100 & 96 & 97 & 98 & 99 & 99 & 99 & 99 & 99 & 99 \\
   \hline
\end{tabular}
\caption{\label{Table4} Percentage of 10 000 MC samples declared significant by various
tests for the Cauchy distribution using ML estimation ($\alpha=0.05, n=50$)}
\end{center}
\end{sidewaystable}

The main conclusions that can be drawn from the simulation results are the following:
\begin{itemize}
\item
As always in similar situations, there exists no uniformly most powerful test, which is in accordance to the results in \cite{J:2000}.
\item
As expected, the power of nearly all test statistics increases for increasing sample size for all alternative distributions. An exception is the Kullback-Leibler distance based test. Its power decreases for certain alternatives, in particular for the Mittag-Leffler distribution. 
%This strange behavior may be due to the choice of  the window size.
\item
The new tests $T_{n,a}, a=1,\ldots,6$, perform better using the maximum likelihood estimator than with median and half-IQR. Hence, the latter estimators should not be used. For all other test statistics, including $T_{n,0}$, performance is comparable, or even better when using median and half-IQR. In any case, the choice of the estimation method can have a pronounced influence on the performance of the tests.
\item
For alternatives with finite first and second moments, the Kullback-Leibler distance based test has the highest power among all competitors. On the other hand, its power breaks down completely for some alternatives without existing first moment, and it is low for some  alternatives with infinite second moment. Hence, the test can not really be seen as an omnibus test.
\item
Among the new tests, values of the tuning parameter around $a=3$ result in a quite homogeneous power against all alternatives. For the tests based on $D_{n,\lambda}$, $\lambda=5$ seems to be a good choice. Both classes of tests perform similarly for the ML estimator; for the median and IQR estimator, the latter is preferable. 
\item 
Among the group of edf tests, $W$ outperform the other tests in most cases.
\item 
For symmetric alternatives without existing first moment as Stable(0.4,0), Stable(0.4,0), Tukey-L(-3) and Tukey-L(-2), the tests based on $T_{n,0}, AD$ and $D_{n,6}$ perform best.
\end{itemize}

\section{Real data example: log-returns of cryptocurrencies}

In this section, we apply the tests for the Cauchy distribution to log-returns of various cryptocurrencies, namely
Bitcoin (BTC), Ethereum (ETH), Ripple (XRP), Litecoin (LTC), BitcoinCash (BCH), EOS (EOS), BinanceCoin (BNB) and Tron(TRX). 
The Cauchy distribution is found to be a comparably good model for such data sets in \cite{S:2020}. There, 58 hypothetical distributions have been fitted to 15 major cryptocurrencies, and the Cauchy model turned out to be the best fitting distribution for 10 cryptocurrencies (including all currencies considered here). In \cite{S:2020}, the number of observations of the various data sets varied widely from 638 to 2255. Further, with very large data sets, each model will be rejected in the end. Hence, we decided to consider shorter time series: a series with daily observations from January 01, 2020 to June 10, 2021, with sample size 527 (526 for ETH), and an even shorter series from January 01, 2021 to June 10, 2021, with sample size 161 (160 for ETH). All prices are closing values in U.S. dollars obtained from cryptodatadownload.com. Returns are estimated by taking logarithmic differences. Days with zero trading volume are omitted. Figure \ref{fig:crypto-hist} shows histogramms of the datasets with larger time span, together with the densities of fitted Cauchy distributions. Visually, the Cauchy model seems to be a reasonable approximation.

\begin{figure}
\centering
\includegraphics[scale=0.7]{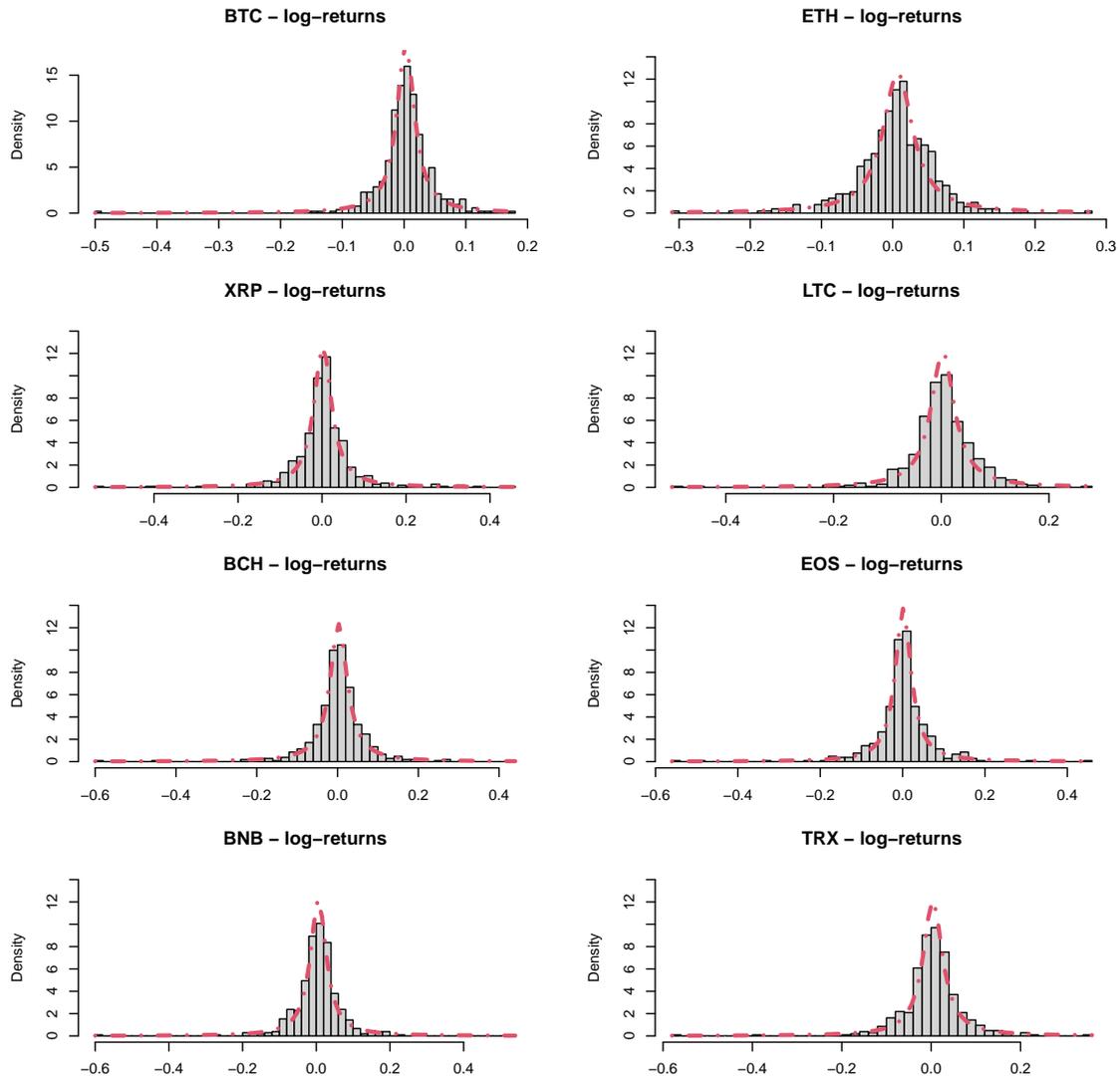}
\caption{Histograms of cryptocurrency log-returns from January\,01,\,2020 to June\,10,\,2021, together with fitted Cauchy densities}    
\label{fig:crypto-hist}
\end{figure}

Tables \ref{Table5} and \ref{Table6} report the results of the different tests for the Cauchy model for the two time series. The $p$-values are based on Monte Carlo simulation under the Cauchy model with $10^4$ replications. For the test based on the Kullback-Leibler distance, we choose the window length $m=100$ for the longer time series, and $m=50$ for the shorter one. 
The results show that even if the Cauchy distribution fits better than many other distributional models, it is still not an acceptable fit in any of the cases. The tests based on $T_{n,5}, KL$ and $D_{n,5}$ result in $p$-values of 0.000 for all currencies for the longer time series, and $p$-values smaller than 0.01 for all currencies for the shorter one. The edf based tests have generally larger $p$-values, with the Watson test having smallest, and Cram\'{e}r-von Mises test having largest $p$-values.

Overall, the EOS data seem to be most compatible with the hypothesis of a Cauchy distribution. For the shorter time series, the $p$-values of all edf based tests are larger than 0.1, and the Kolmogorov-Smirnov and Cram\'{e}r-von Mises tests don't reject $H_0$ even for the larger data set. These tests also show relatively large values for Ripple (XRP) and BitcoinCash (BCH). On the whole, the results confirm the findings of the simulation study concerning the comparative power of the test.

\begin{table}
\begin{center}
\begin{tabular}{rrrrrrrrrrrrr}
  \hline
& $T_{n,1}$ & $T_{n,3}$ & $T_{n,5}$ & $T_{n,0}$ & $KL$ & $KS$ & $CM$ & $AD$ & $W$ & $D_{n,1}$ & $D_{n,3}$ & $D_{n,5}$ \\ 
 \hline
BTC &
0.122&      0.000&      0.000&      0.001&0.000&0.005&0.066&0.003&0.000&0.000&      0.000&      0.000 \\
ETH &
0.000&      0.000&      0.000&      0.000&0.000&0.000&0.012&0.000&0.000&0.000&      0.000&      0.000 \\      
XRP &
0.357&      0.000&      0.000&      0.016&0.000&0.064&0.191&0.019&0.004&0.003&      0.000&      0.000 \\    
LTC &
0.008&      0.000&      0.000&      0.000&0.000&0.000&0.023&0.001&0.000&0.000&      0.000&      0.000 \\      
BCH &
0.036&      0.000&      0.000&      0.000&0.000&0.035&0.063&0.004&0.000&0.000&      0.000&      0.000\\    
EOS &
0.308&      0.009&      0.000&      0.559&0.000&0.119&0.277&0.044&0.024&0.006&      0.000&      0.000 \\  
BNB &
0.001&      0.000&      0.000&      0.000&0.000&0.002&0.029&0.002&0.000&0.000&      0.000&      0.000 \\      
TRX &
0.008&      0.000&      0.000&      0.001&0.000&0.015&0.042&0.002&0.000&0.000&      0.000&      0.000 \\    
\hline
\end{tabular}
\caption{\label{Table5} $p$-values of cryptocurrency log-returns (Jan.\,01,\,2020-June\,10,\,2021) for various tests for the Cauchy distribution}
\end{center}
\end{table}

\begin{table}
\begin{center}
\begin{tabular}{rrrrrrrrrrrrr}
  \hline
& $T_{n,1}$ & $T_{n,3}$ & $T_{n,5}$ & $T_{n,0}$ & $KL$ & $KS$ & $CM$ & $AD$ & $W$ & $D_{n,1}$ & $D_{n,3}$ & $D_{n,5}$ \\ 
 \hline
BTC &
0.134 &     0.000 &     0.000 &     0.019&0.000&0.062&0.151&0.041&0.005&0.002&      0.000&      0.000 \\
ETH &
0.000&      0.000&      0.000&      0.000&0.000&0.004&0.022&0.006&0.000&0.000&      0.000&      0.000 \\      
XRP &
0.216&      0.003&      0.001&      0.014&0.000&0.147&0.247&0.090&0.010&0.016&      0.002&      0.002   \\    
LTC &
0.132&      0.000&      0.000&      0.015&0.000&0.073&0.163&0.043&0.002&0.001&      0.000&      0.000 \\      
BCH &
0.049&      0.001&      0.001&      0.001&0.001&0.316&0.156&0.068&0.004&0.011&      0.002&      0.001   \\    
EOS &
0.387&      0.011&      0.001&      0.339&0.001&0.222&0.458&0.153&0.108&0.021&      0.005&      0.006     \\  
BNB &
0.023&      0.000&      0.000&      0.032&0.001&0.020&0.074&0.026&0.000&0.000&      0.000&      0.000 \\      
TRX &
0.010&      0.000&      0.000&      0.000&0.000&0.070&0.074&0.027&0.000&0.001&      0.000&      0.000   \\    
\hline
\end{tabular}
\caption{\label{Table6} $p$-values of cryptocurrency log-returns (Jan.\,01,\,2021-June\,10,\,2021), for various tests for the Cauchy distribution}
\end{center}
\end{table}

\section{Outlook}
%We have proposed a family of tests for the Cauchy family of distributions with desirable theoretical and computational features as consistency and competitiveness to existing procedures. The behaviour of the tests have been derived under fixed alternatives.

We conclude the paper by pointing out some open problems for further research. The explicit formula of the covariance kernels in Remark \ref{rem:est} open ground to numerical approximation of the eigenvalues of the integral operator $\mathcal{K}$. Especially an approximation of the largest eigenvalue would offer more theoretical insights, since, as is well known, it is useful for efficiency statements in the sense of Bahadur, see \cite{B:1960} and \cite{N:1995}. A profound knowledge of the eigenvalues leads to explicit values of cumulants of the null distribution, such that a fit of the null distribution to the Pearson distribution system is possible. This usually gives a very accurate approximation of the quantiles of the distribution and hence of the critical values of the test, see i.e. in the context of normality testing \cite{E:2020,H:1990}. Another desirable extension for the family of tests is a data dependent choice of the tuning parameter $a$, see \cite{T:2019} for a procedure which might be applicable. 

\bibliographystyle{abbrv}
\bibliography{lit-Cauchy}  

\begin{appendix}
\section{Proofs} 
\subsection{Proof of Theorem \ref{thm:asympVer}}\label{app:proofasy1}
A first order multivariate Taylor expansion around $(\alpha,\beta)=(0,1)$ and the definition of $Y_{n,j}$ in (\ref{Y}) leads to
\begin{align}
\label{Z2}
Z_n^*(t)=\frac{1}{\sqrt{n}}\sum_{j=1}^n\bigg[&\bigg(\frac{2X_j}{1+X_j^2}+t\bigg)\cos(tX_j)+\bigg(t-\frac{2X_j}{1+X_j^2}\bigg)\sin(tX_j)\nonumber\\\
\begin{split}
+&\Bigg(\bigg(\frac{2X_j}{1+X_j^2}t-\Big(t^2+\frac{2(1-X_j^2)}{(1+X_j^2)^2}\Big)\bigg)\cos(tX_j)\\
&+\bigg(\frac{2X_j}{1+X_j^2}t+\Big(t^2+\frac{2(1-X_j^2)}{(1+X_j^2)^2}\Big)\bigg)\sin(tX_j)\Bigg)\widehat{\alpha}_n
\end{split}\\\
\begin{split}
+&X_j\Bigg(\bigg(\frac{2X_j}{1+X_j^2}t-\Big(t^2+\frac{2(1-X_j^2)}{(1+X_j^2)^2}\Big)\bigg)\cos(tX_j)\\
&+\bigg(\frac{2X_j}{1+X_j^2}t+\Big(t^2+\frac{2(1-X_j^2)}{(1+X_j^2)^2}\Big)\bigg)\sin(tX_j)\Bigg)(\widehat{\beta}_n-1)\bigg],\quad t\in\R.
\end{split}\nonumber
\end{align}
Since $\widehat{\alpha}_n$ and $\widehat{\beta}_n$ are consistent estimators, we have $(\widehat{\alpha}_n,\widehat{\beta}_n)\overset{\PP}{\rightarrow}(0,1)$ and hence
\[Z_n(t)=Z_n^*(t)+o_{\PP}(1),\quad t\in\R,\]
and by the triangle inequality as well as Slutsky's theorem we have
\[\Vert Z_n-Z_n^*\Vert_{\mathbb{H}}\overset{\PP}{\longrightarrow}0\quad\mbox{as}\;n\rightarrow\infty.\]
Note that by assumption, we have
\begin{equation}
\label{BR}
\sqrt{n}\widehat{\alpha}_n=\frac{1}{\sqrt{n}}\sum_{j=1}^n\psi_1(X_j)+r_{1n},\ \ \sqrt{n}(\widehat{\beta}_n-1)=\frac{1}{\sqrt{n}}\sum_{j=1}^n\psi_2(X_j)+r_{2n},
\end{equation}
with $r_{1n},r_{2n}=o_\PP(1)$ and $\E[\psi_1(X)]=\E[\psi_2(X)]=0$ and $\E[\psi(X)\psi(X)^\top]=C\cdot I_{2}$,  where $X\sim \mbox{C}(0,1)$, $\psi(X)=(\psi_1(X),\psi_2(X))$, $C$ is a positive constant and $I_{2}$ and is the $2\times2$ unit matrix. Denoting
\begin{equation}
\label{g}
g(t,x)=\bigg(\frac{2x}{1+x^2}t-\Big(t^2+\frac{2(1-x^2)}{(1+x^2)^2}\Big)\bigg)\cos(tx)+\bigg(\frac{2x}{1+x^2}t+\Big(t^2+\frac{2(1-x^2)}{(1+x^2)^2}\Big)\bigg)\sin(tx),
\end{equation}
we see with symmetry arguments
\[\E[g(t,X_1)]=-\frac{1}{2}\big(t^2+\vert t\vert+1\big)e^{-\vert t\vert},\quad\text{and}\quad\E[X_1g(t,X_1)]=\frac{1}{2}t\big(\vert t\vert+1\big)e^{-\vert t\vert},\]
By the law of large numbers in Hilbert spaces, we have
\begin{equation}
\label{GGZ1}
\Big\Vert\frac{1}{n}\sum_{j=1}^ng(\cdot,X_j)-\E[g(\cdot,X_1)]\Big\Vert_{\mathbb{H}}=o_{\PP}(1)
\end{equation}
and
\begin{equation}
\label{GGZ2}
\Big\Vert\frac{1}{n}\sum_{j=1}^nX_jg(\cdot,X_j)-\E[X_1g(\cdot,X_1)]\Big\Vert_{\mathbb{H}}=o_{\PP}(1).
\end{equation}
Replacing $\sqrt{n}\widehat{\alpha}_n$ and $\sqrt{n}(\widehat\beta_n-1)$ by the linear representations (\ref{BR}), and substituting the arithmetic means by their asymptotic counterparts leads to
\begin{align}
\label{Z3}
\widetilde{Z}_n(t)=&\frac{1}{\sqrt{n}}\sum_{j=1}^n\bigg[\bigg(\frac{2X_j}{1+X_j^2}+t\bigg)\cos(tX_j)+\bigg(t-\frac{2X_j}{1+X_j^2}\bigg)\sin(tX_j)\\\
&-\frac{1}{2}\big(t^2+\vert t\vert+1\big)e^{-\vert t\vert}\psi_1(X_j)+\frac{1}{2}t\big(\vert t\vert+1\big)e^{-\vert t\vert}\psi_2(X_j)\bigg],\quad t\in\R.\nonumber
\end{align}
Since for all $a_1,a_2\in\R$,  $(a_1+a_2)^2\le2a_1^2+2a_2^2$ holds, we have
\begingroup
\allowdisplaybreaks
\begin{align*}
\Vert Z_n^*-\widetilde{Z}_n\Vert_{\mathbb{H}}&= \int_{-\infty}^{\infty}\bigg\vert\frac{1}{\sqrt{n}}\sum_{j=1}^ng(t,X_j)\widehat{\alpha}_n+X_jg(t,X_j)\widehat{\beta}_n+\frac{1}{2}\big(t^2+\vert t\vert+1\big)e^{-\vert t\vert}\psi_1(X_j)\\
&\ \ \ \ \ -\frac{1}{2}t\big(\vert t\vert+1\big)e^{-\vert t\vert}\psi_2(X_j)\bigg\vert^2\omega(t)dt\\
&\overset{(\ref{BR})}{\leq}\int_{-\infty}^{\infty}\bigg\vert\bigg(\frac{1}{n}\sum_{j=1}^ng(t,X_j)+\frac{1}{2}\big(t^2+\vert t\vert+1\big)e^{-\vert t\vert}\bigg)\frac{1}{\sqrt{n}}\sum_{k=1}^n\psi_1(X_k)\\
&\ \ \ \ \ +\bigg(\frac{1}{n}\sum_{j=1}^nX_jg(t,X_j)-\frac{1}{2}t\big(\vert t\vert+1\big)e^{-\vert t\vert}\bigg)\frac{1}{\sqrt{n}}\sum_{k=1}^n\psi_2(X_k)\bigg\vert^2\omega(t)dt+o_{\PP}(1)\\
&\leq2\bigg\Vert\frac{1}{n}\sum_{j=1}^ng(\cdot,X_j)-\E[g(\cdot,X_1)]\bigg\Vert^2_{\mathbb{H}}\bigg(\frac{1}{\sqrt{n}}\sum_{k=1}^n\psi_1(X_k)\bigg)^2\\
&\ \ \ \ \ +2\bigg\Vert\frac{1}{n}\sum_{j=1}^nX_jg(\cdot,X_j)-\E[X_1g(\cdot,X_1)]\bigg\Vert^2_{\mathbb{H}}\bigg(\frac{1}{\sqrt{n}}\sum_{k=1}^n\psi_2(X_k)\bigg)^2+o_{\PP}(1).
\end{align*}
\endgroup

By (\ref{GGZ1}) and (\ref{GGZ2}), the tightness of $\frac{1}{\sqrt{n}}\sum_{j=1}^n\psi_i(X_j),i=1,2$ and Slutzky's Theorem, we have 
\[\Vert Z_n^*-\widetilde{Z}_n\Vert\overset{\PP}{\longrightarrow}0,\quad\mbox{as}\;n\rightarrow\infty.\]
Hence $Z_n$ and $\widetilde{Z}_n$ share the same weak limit. Writing
\[\widetilde{Z}_n(t)=\frac{1}{\sqrt{n}}\sum_{j=1}^n\widetilde{Z}_{n,j}(t),\ t\in\R,\]
where
\begin{align*}
\widetilde{Z}_{n,j}(t)= & \bigg(\frac{2X_j}{1+X_j^2}+t\bigg)\cos(tX_j)+\bigg(t-\frac{2X_j}{1+X_j^2}\bigg)\sin(tX_j)\\
&-\frac{1}{2}\big(t^2+\vert t\vert+1\big)e^{-\vert t\vert}\psi_1(X_j)+\frac{1}{2}t\big(\vert t\vert+1\big)e^{-\vert t\vert}\psi_2(X_j),\ t\in\R,
\end{align*}
$j=1,...,n,$ we have $\E[\widetilde{Z}_{n,1}(t)]=0$ and $\widetilde{Z}_{n,1},\widetilde{Z}_{n,2},...$ are i.i.d. centred random elements in $\mathbb{H}$. Straightforward calculations show that $K(s,t)=\E[\widetilde{Z}_{n,1}(s)\widetilde{Z}_{n,1}(t)]$ has the stated formula and the assertion follows by the central limit theorem in Hilbert spaces.\hfill$\square$

%----------------------------------------------------------------------------------------------------------

\subsection{Proof of Theorem \ref{thm:consistency}}\label{app:cons}
Let $Z_n$ and $Z_n^*$ be defined as in (\ref{eq:Zn}) and (\ref{Z2}). Since by assumption we have  $(\widehat{\alpha}_n,\widehat{\beta}_n)\overset{\PP}{\longrightarrow}(0,1)$, we can show in complete analogy to the proof of Theorem \ref{thm:asympVer}, that $\Vert n^{-1/2}(Z_n-Z_n^*)\Vert_{\mathbb{H}}\overset{\PP}{\longrightarrow}0$ holds. Let $g(\cdot,\cdot)$ be defined as in (\ref{g}) and put
\[Z_n^0(t)=\frac{1}{\sqrt{n}}\sum_{j=1}^n\bigg(\frac{2X_j}{1+X_j^2}+t\bigg)\cos(tX_j)+\bigg(t-\frac{2X_j}{1+X_j^2}\bigg)\sin(tX_j),\quad t\in\R.\]
Then
\begin{align*}
n^{-1/2}\big(Z_n^*(t)-Z_n^0(t)\big)& = \frac{1}{n}\sum_{j=1}^n\bigg[\bigg(\frac{2X_j}{1+X_j^2}t-\bigg(t^2+\frac{2(1-X_j^2)}{(1+X_j^2)^2}\bigg)\bigg)\cos(tX_j)\\
&\ \ \ \ \ \ \ \ \ \ \ +\bigg(\frac{2X_j}{1+X_j^2}t+\bigg(t^2+\frac{2(1-X_j^2)}{(1+X_j^2)^2}\bigg)\sin(tX_j)\bigg]\widehat{\alpha}_n\\
&\ \ \ \ \ \ +X_j\bigg[\bigg(\frac{2X_j}{1+X_j^2}t-\bigg(t^2+\frac{2(1-X_j^2)}{(1+X_j^2)^2}\bigg)\bigg)\cos(tX_j)\\
&\ \ \ \ \ \ \ \ \ \ \ +\bigg(\frac{2X_j}{1+X_j^2}t+\bigg(t^2+\frac{2(1-X_j^2)}{(1+X_j^2)^2}\bigg)\sin(tX_j)\bigg](\widehat{\beta}_n-1)\\
&=\widehat{\alpha}_n\frac{1}{n}\sum_{j=1}^ng(t,X_j)+(\widehat{\beta}_n-1)\frac{1}{n}\sum_{j=1}^nX_jg(t,X_j)
\end{align*}
follows and by the triangle inequality we have
\begin{align}
\label{Zn-Z0}
\Vert n^{-1/2}(Z_n^*-Z_n^0)\Vert_{\mathbb{H}}^2&=\int_{-\infty}^{\infty}\Big\vert\widehat{\alpha}_n\frac{1}{n}\sum_{j=1}^ng(t,X_j)+(\widehat{\beta}_n-1)\frac{1}{n}\sum_{j=1}^nX_jg(t,X_j)\Big\vert^2\omega(t)dt\nonumber\\
&\leq\int_{-\infty}^{\infty}\bigg(\vert\widehat{\alpha}_n\vert\Big\vert\frac{1}{n}\sum_{j=1}^ng(t,X_j)\Big\vert+\vert\widehat{\beta}_n-1\vert\Big\vert\frac{1}{n}\sum_{j=1}^nX_jg(t,X_j)\Big\vert\bigg)^2\omega(t)dt\nonumber\\
&\leq2\vert\widehat{\alpha}_n\vert^2\Big\Vert\frac{1}{n}\sum_{j=1}^ng(t,X_j)\Big\Vert_{\mathbb{H}}^2+2\vert\widehat{\beta}_n-1\vert^2\Big\Vert\frac{1}{n}\sum_{j=1}^nX_jg(t,X_j)\Big\Vert_{\mathbb{H}}^2\nonumber
\end{align}
By the law of large numbers in Hilbert spaces and $(\widehat{\alpha}_n,\widehat{\beta}_n)\overset{\PP}{\longrightarrow}(0,1)$, the right hand side is $o_\PP(1)$. Note that the existence of the mean values is guaranteed by the assumptions. Again, the law of large numbers in $\mathbb{H}$ shows
\[n^{-1/2}Z_n^0(t)\overset{a.s.}{\longrightarrow}\E\bigg[\bigg(\frac{2X}{1+X^2}+t\bigg)\cos(tX)+\bigg(t-\frac{2X}{1+X^2}\bigg)\sin(tX)\bigg],\quad\mbox{as}\;n\rightarrow\infty,\]
and by the symmetry of the weight function, we have 
\begin{align}
\label{Delta}
\frac{T_n}{n} = \Big\Vert\frac{1}{\sqrt{n}}\ Z_n\Big\Vert_{\mathbb{H}}^2\nonumber&\overset{\PP}{\longrightarrow}\int_{-\infty}^{\infty}\bigg\vert\E\bigg[\bigg(\frac{2X}{1+X^2}+t\bigg)\cos(tX)+\bigg(t-\frac{2X}{1+X^2}\bigg)\sin(tX)\bigg]\bigg\vert^2\omega(t)dt\\
&\ \ \ =\int_{-\infty}^{\infty}\bigg\vert\E\bigg[\bigg(it-\frac{2X}{1+X^2}\bigg)e^{itX}\bigg]\bigg\vert^2\omega(t)dt,\nonumber
\end{align}
as $n\rightarrow\infty$.\hfill$\square$

%----------------------------------------------------------------------------------------------------------

\subsection{Proof of Lemma \ref{lem:Hilfslemma}}\label{app:PL}
Let
\begin{align*}
\widetilde{Z}_n^{\bullet}(t) =\frac{1}{\sqrt{n}}\sum_{j=1}^n\bigg[&\bigg(\frac{2X_j}{1+X_j^2}+t\bigg)\cos(tX_j)+\bigg(t-\frac{2X_j}{1+X_j^2}\bigg)\sin(tX_j)\\
&+\E[g(t,X)]\psi_1(X_j)+\E[Xg(t,X)]\psi_2(X_j)\bigg].
\end{align*}
By a first order Taylor approximation of $Z_n$ in (\ref{eq:Zn}) around $(\alpha,\beta)=(0,1)$and the same steps as in the proof of Theorem \ref{thm:asympVer}, we have $\Vert Z_n-\widetilde{Z}_n^{\bullet}\Vert_{\mathbb{H}}=o_\PP(1)$. Putting $\widetilde{W}_n(t)=\sqrt{n}\big(\frac{\widetilde{Z}_n^{\bullet}(t)}{\sqrt{n}}-z(t)\big)$, $t\in\R$, we have

\begin{align*}
\Vert W_n-\widetilde{W}_n\Vert_{\mathbb{H}} & =\Big\Vert \sqrt{n}\Big(Z_n^\bullet-z\Big)-\sqrt{n}\Big(\frac{\widetilde{Z}_n^\bullet}{\sqrt{n}}-z\Big)\Big\Vert_\mathbb{H}=\Big\Vert \sqrt{n}\Big(\frac{Z_n}{\sqrt{n}}-\frac{\widetilde{Z}_n^\bullet}{\sqrt{n}}\Big)\Big\Vert_\mathbb{H}=\Vert Z_n-\widetilde{Z}_n^\bullet\Vert_\mathbb{H},
\end{align*}
and by the central limit theorem in Hilbert spaces
\[\sqrt{n}\bigg(\frac{\widetilde{Z_n^{\bullet}}}{\sqrt{n}}-z\bigg)=\widetilde{W}_n\overset{\mathcal{D}}{\longrightarrow}\mathcal{W}, \mbox{as}\; n\rightarrow\infty,\]
follows. Here, $\mathcal{W}\in\mathbb{H}$ is a centred Gaussian process with covariance kernel $K(s,t)=\E[(\widetilde{Z}_1^{\bullet}(s)-z(s))(\widetilde{Z}_1^{\bullet}(t)-z(t))]$. The stated representation of $K$ follows by direct calculations.\hfill $\square$
\end{appendix}
\end{document}